\documentclass[12pt,reqno]{amsart} 
\usepackage[T1]{fontenc}
\usepackage{amsfonts}
\usepackage{amssymb}
\usepackage{mathrsfs}
\usepackage[latin1]{inputenc}
\usepackage{amsmath}
\usepackage{paralist}
\usepackage[english]{babel}
\usepackage{setspace}
\usepackage{bbm}

\newtheorem{theorem}{Theorem}[section]
\newtheorem{lemma}[theorem]{Lemma}

\newtheorem{remark}[theorem]{Remark}
\newtheorem{prop}[theorem]{Proposition}
\newtheorem{example}[theorem]{Example}
\newtheorem{corollary}[theorem]{Corollary}

\numberwithin{equation}{section}

\newcommand{\D}{{\mathbb D}}

\newcommand{\C}{{\mathbb C}}
\newcommand{\N}{{\mathbb N}}

\newcommand{\cL}{{\mathcal L}}

\newcommand{\cP}{{\mathcal P}}

\newcommand{\cB}{{\mathcal B}}

\newcommand{\su}{\subseteq}

\newcommand\Ker{\mathop{\rm Ker}}

\newcommand{à}{\`a}

\begin{document}

\title{Optimal domain of Volterra operators in classes of Banach
	spaces of analytic functions}

\author{Angela\,A. Albanese, Jos\'e Bonet and Werner\,J. Ricker}


\address{ Angela A. Albanese\\
Dipartimento di Matematica e Fisica
``E. De Giorgi''\\
Universit\`a del Salento- C.P.193\\
I-73100 Lecce, Italy}
\email{angela.albanese@unisalento.it}

\address{Jos\'e Bonet \\
Instituto Universitario de Matem\'{a}tica Pura y Aplicada
IUMPA \\
Universitat Polit\`ecnica de Val\`encia \\
E-46071 Valencia, Spain} \email{jbonet@mat.upv.es}

\address{Werner J.  Ricker \\
Math.-Geogr. Fakultät \\
 Katholische Universität
Eichst\"att-Ingol\-stadt \\
D-85072 Eichst\"att, Germany}
\email{werner.ricker@ku.de}

\begin{abstract}
A thorough investigation is made of the optimal domain space of generalized Volterra operators, Cesàro operators and other operators when they act in various Banach spaces of analytic functions. Of particular interest is the situation when the operators act in Hardy spaces, Korenblum growth spaces and more general weighted spaces. The optimal domain space may be genuinely larger than the initial domain of the operator, or not. In the former case, the initial space may or may  not be dense in the optimal domain space. Sometimes the optimal domain space can be identified with a known Banach space of analytic functions, on other occasions it determines a new space. 
\end{abstract}

\maketitle

\markboth{A.\,A. Albanese, J. Bonet and W.\,J. Ricker}%
{\MakeUppercase{Optimal domain }}

\section{Introduction}

Let $\D:=\{z\in\C:\ |z|<1\}$ and $H(\D)$ denote the Fr\'echet space consisting of all analytic functions on $\D$ equipped with the topology $\tau_c$ of uniform convergence on the compact subsets of $\D$. The classical Cesàro operator $C$, given by 
\[
(Cf)(z):=\sum_{n=0}^\infty\left(\frac{1}{n+1}\sum_{k=0}^n\hat{f}(k)\right)z^n,\quad z\in\D,
\]
where $f\in H(\D)$ has Taylor coefficients $(\hat{f}(n))_{n\in\N_0}:=(\frac{f^{(n)}(0)}{n!})_{n\in\N_0}$, is known to be a continuous linear operator on $H(\D)$ and on the Hardy spaces $H^p$ over $\D$, for $1\leq p<\infty$; see, for example, \cite{Sis}, \cite{Sis2} and the references therein. It was observed in \cite{CR} that, for each $1\leq p<\infty$, there exist certain weighted Hardy spaces $H^p(\omega)$ satisfying $H^p\varsubsetneqq H^p(\omega)$ such that $C\colon H^p(\omega)\to H^p$ continuously. This led to the notion of the \textit{largest} Banach space $X$ of analytic functions on $\D$ continuously contained in $H(\D)$ such that $H^p\subseteq X$ and $C(X)\subseteq H^p$. This largest space $X$, which always exists, was called the \textit{optimal domain} of $C$ (for $H^p$) and was denoted by $[C,H^p]$. The space $[C,H^p]$,  itself  a Banach space of analytic functions on $\D$, is thoroughly investigated in \cite{CR}.

This topic  was  taken up again in the recent article \cite{BDNS}. Given $g\in H(\D)$, consider the generalized Volterra operator 
\begin{equation}\label{eq.GVO}
	(T_gf)(z):=\int_0^zf(\xi)g'(\xi)\,d\xi,\quad z\in\D,
\end{equation}
for $f\in H(\D)$. When $g(z)=z$ on $\D$, this operator is the classical Volterra operator. It is known that $T_g\colon H^p\to H^p$, for $1\leq p<\infty$, is continuous if and only if $g\in BMOA$; see \cite{Pom} for $p=2$ and \cite{AleSi1} for $1\leq p<\infty$. For $g(z)=-{\rm Log} (1-z)$ one recovers the Cesàro operator $C$ in the sense that
\[
(T_gf)(z)=\int_0^z \frac{f(\xi)}{1-\xi}\,d\xi=z(Cf)(z),\quad z\in\D,\quad f\in H^p.
\]
The notation used in \cite{BDNS} for the optimal domain is, of course, adapted from \cite{CR}, namely,
\[
[T_g, H^p]:=\{f\in H(\D):\ T_gf\in H^p\}
\]
equipped with the norm $\|f\|_{[T_g,H^p]}:=\|T_gf\|_{H^p}$. It is shown, for each $1\leq p<\infty$ and $g\in BMOA$, that $[T_g, H^p]$ is actually  a Banach space of analytic functions on $\D$, that the continuous inclusion $H^p\subseteq [T_g,H^p]$ is \textit{strict}, that
\begin{equation}\label{eq.UOP}
	H^p=\cap_{g\in BMOA}[T_g,H^p]
	\end{equation}
and,  that $[T_g,H^{p_2}]\varsubsetneqq [T_g,H^{p_1}]$ whenever $1\leq p_1<p_2<\infty$. Each of the articles \cite{BDNS}, \cite{CR} deal with the optimal domain of operators acting in the \textit{particular class} of spaces $H^p$. In \cite{CR} the relevant operator is the Cesàro operator whereas \cite{BDNS} deals with the entire class of generalized Volterra operators $T_g$, which contains $C$ as one of its members. It is the \textit{injectivity} of $T_g$ in $H(\D)$, whenever $g\in H(\D)$ is non-constant (cf. \cite{AND}), which ensures that the semi-norm $\|\cdot\|_{[T_g,H^p]}$ is actually a \textit{norm} in $[T_g,H^p]$. For the completeness of $[T_g,H^p]$ see \cite[Theorem 1]{BDNS}. At this stage it is important to point out that the notation used for generalized Volterra operators in the literature is not consistent. In this article and in \cite{ABR_JFA}  we will (as do many other authors) denote the above operator $T_g$ by $V_g$, whereas $T_g$ will henceforth denote the related operator given by
\begin{equation}\label{eq.CC}
	(T_gf)(z):=\frac{1}{z}\int_0^zf(\xi)g'(\xi)\,d\xi,\quad z\in\D\setminus\{0\},\ (T_gf)(0):=f(0)g'(0),
\end{equation}
for $f\in H(\D)$. That is, $V_g$ denotes the operator on the right-side of \eqref{eq.GVO} whereas $T_g$ is defined by \eqref{eq.CC}.

The first appearance of a study of the optimal domain space of an operator acting in a Banach space of analytic functions $X$ on $\D$ \textit{other} than a $H^p$-space occurs in \cite{CR1}. The operator involved there is again the Cesàro operator $C$ but, its action is now in the  space
\[
X=\mathscr{H}(ces_2):=\left\{f\in H(\D):\ \left(\frac{1}{n+1}\sum_{k=0}^n|\hat{f}(k)|\right)_{n\in\N_0}\in\ell^2\right\}
\]
equipped with the norm 
\[
\|f\|_{\mathscr{H}(ces_2)}:=\left\|\left(\frac{1}{n+1}\sum_{k=0}^n|\hat{f}(k)|\right)_{n\in\N_0}\right\|_{\ell^2},\quad f\in \mathscr{H}(ces_2).
\]
It is shown in \cite{CR1}, amongst other things,  that $\mathscr{H}(ces_2)$ is a Banach space of analytic functions on $\D$, that $C\colon \mathscr{H}(ces_2)\to \mathscr{H}(ces_2)$ is continuous and that
\[
H^2\varsubsetneqq \mathscr{H}(ces_2)\varsubsetneqq [C,H^2]\varsubsetneqq [C, \mathscr{H}(ces_2)],
\] 
with all inclusions continuous.

One of the aims of \cite{ABR_JFA} is to treat larger classes of Banach spaces of analytic functions on $\D$ beyond the spaces $H^p$. For example, $X$ can be one of the well known \textit{Korenblum  growth Banach spaces} $A^{-\gamma}$ and $A^{-\gamma}_0$ of order  $\gamma>0$. It turns out that the relevant class of symbols $g$ for the operators $V_g$ acting in $A^\gamma$ (resp. $A^\gamma_0$) is the Bloch space $\cB$ (resp. $\cB_0$); see \cite[Section 3]{ABR_JFA}. Alternate descriptions of $[V_g,A^{-\gamma}]$ (resp. $[V_g,A^{-\gamma}_0]$) are presented in Proposition 3.5 and Remark 3.6. To see that the inclusion $A^{-\gamma}\subseteq [V_g, A^{-\gamma}]$ can be strict we refer to Example 3.7 in \cite{ABR_JFA}. The analogue of \eqref{eq.UOP} occurs in Propositions 3.8 and 3.9, namely,
\[
A^{-\gamma}=\cap_{g\in\cB}[V_g,A^{-\gamma}]\ \mbox{ and }\ A^{-\gamma}_0=\cap_{g\in\cB}[V_g,A^{-\gamma}_0].
\]
In Section 5 of \cite{ABR_JFA}, for a given  $g\in H(\D)$, the attention is on the optimal domain spaces of $V_g$ and $T_g$ acting in certain \textit{weighted Banach spaces} of analytic functions $H^\infty_v$ and $H^0_v$ on $\D$ (equipped with the sup-norm) for a class of weight functions $v\colon [0,1)\to (0,\infty)$. This family of spaces properly contains the Korenblum growth Banach spaces, which correspond to the weight $v_\gamma(r):=(1-r)^\gamma$ on $[0,1)$, for $\gamma>0$. The operators $V_g$ and $T_g$ are continuous on $A^{-\gamma}$ precisely when $g\in \cB$. Given a non-constant function $g\in H(\D)$, the operators $V_g$ and $T_g$ are injective on $H(\D)$ and so $\|\cdot\|_{[V_g,H^\infty_v]}$ and $\|\cdot\|_{[T_g,H^\infty_v]}$ are \textit{norms}, as are $\|\cdot\|_{[V_g,H^0_v]}$ and $\|\cdot\|_{[T_g,H^0_v]}$, whenever $V_g(H^\infty_v)\subseteq H^\infty_v$ and $T_g(H^\infty_v)\subseteq H^\infty_v$ (resp. $V_g(H^0_v)\subseteq H^0_v$ and $T_g(H^0_v)\subseteq H^0_v$). For each $g\in H(\D)$ it is shown that $V_g\colon H^\infty_v\to H^\infty_v$ (resp. $V_g\colon H^0_v\to H^0_v$) is continuous if and only if $T_g\colon H^\infty_v\to H^\infty_v$ (resp. $T_g\colon H^0_v\to H^0_v$) is continuous  and, if this is the case, then $[V_g,H^\infty_v]=[T_g,H^\infty_v]$ (resp. $[V_g,H^0_v]=[T_g,H^0_v]$). For $g_0(z):=-{\rm Log}(1-z)$ on $\D$,  which satisfies $g\in \cB_0$, the operator $T_{g_0}=C$. For each $\gamma>0$, it is shown that
\[
A^{-\gamma}\varsubsetneqq [C,A^{-\gamma}]\ \mbox{ and }\ A^{-\gamma}\varsubsetneqq [V_{g_0},A^{-\gamma}]\
\]
as well as
\[
A^{-\gamma}_0\varsubsetneqq [C,A^{-\gamma}_0]\ \mbox{ and }\ A^{-\gamma}_0\varsubsetneqq [V_{g_0},A^{-\gamma}_0].
\]
It is also established, for each $\beta>\gamma$, that 
\[
A^{-\beta}\varsubsetneqq [C,A^{-\gamma}]\ \mbox{ and }\ A^{-\beta}_0\varsubsetneqq [C,A^{-\gamma}_0].
\]
For further results we refer to \cite{ABR_JFA}.

The current paper presents significant refinements and extensions of the results in \cite{ABR_JFA} together with many new results. The operators concerned are again $V_g$ and $T_g$ acting in Hardy spaces, as well as  in certain  weighted Banach spaces $H^\infty_v$ and $H^0_v$. However, the classes of symbols $g$ considered are quite different to those in \cite{ABR_JFA} and the weight functions $v$ satisfy alternate conditions to the ones appearing in \cite{ABR_JFA}. The disc algebra $A(\D)\subseteq H^\infty$ is also treated. The idea is to invoke the differentiation operator $D$ and the integration operator $J$ (which always act continuously in $H(\D)$) \textit{provided} they are also continuous between the appropriate weighted spaces $H^\infty_v$ and $H^\infty_w$. This requires restrictions on $v$ and $w$ but, has the advantage that effective new techniques become available.

Section 2 develops various results concerning the general theory of optimal domain spaces $[T,X]$, for $X$ a Banach space of analytic functions on $\D$ and $T\colon H(\D)\to H(\D)$ a continuous linear operator satisfying $T(X)\subseteq X$. These results are needed in the sequel, where $T$ and $X$ are specialized to be certain concrete operators and spaces, respectively.

In Subsection 3.1 of Section 3 the relevant weighted spaces $H^\infty_v$ and $H^0_v$ are defined. Both of these spaces are continuously contained in $H(\D)$. If $v$ satisfies $\lim_{r\to 1^-}v(r)=0$, then $H^0_v$ is the closure in $H^\infty_v$ of the space of all polynomials. An important and well known concept that is needed is that of $v$ being \textit{log-convex}. In the presence of log-convexity of $v$ and for the associated weight $w(r):=(1-r)v(r)$ on $[0,1)$ it is known that $D$ is continuous from $H^\infty_v$ into $H^\infty_w$ if and only if $v$ satisfies the requirement \eqref{eq.Con-D}. If, in addition, $\lim_{r\to 1^-}v(r)=0$, the  same is true of $D$ mapping 
$H^0_v$ into $H^0_w$. As a sample, under these conditions, we mention for any non-constant function $g\in H(\D)$, that the continuity of $D\colon H^\infty_v\to H^\infty_w$ implies that $H^\infty_v$ is a \textit{proper} subspace of the optimal domain space $[V_g,H^0_v]$ whenever $V_g\colon H^0_v\to H^0_v$ is compact; see Proposition \ref{P.H_vcontained}. For the case of the integration operator $J$ it is known, again whenever $v$ is log-convex and $w(r):=(1-r)v(r)$ on $[0,1)$, that the continuity of $J\colon H^\infty_w\to H^\infty_v$ is equivalent to the requirement that \eqref{eq.D} is satisfied. In particular, if the symbol $g\in H(\D)$ satisfies $g', \frac{1}{g'}\in H^\infty$ and the weight $v$ satisfies both \eqref{eq.Con-D} and \eqref{eq.D}, then both operators $D\colon H^\infty_v\to H^\infty_w$ and $J\colon H^\infty_w\to H^\infty_v$ are continuous. As a consequence, whenever $V_g\colon H^\infty_v\to H^\infty_v$ is continuous, the optimal domain spaces are given precisely by 
\begin{equation}\label{eq.14}
	[V_g,H^\infty_v]=[T_g,H^\infty_v]=H^\infty_w
	\end{equation}
with equivalent norms; see Proposition \ref{P.Opt_H_v} and Corollary \ref{C.Opt_H_v}.

For each $\gamma>0$, it was noted above that the Korenblum growth Banach spaces $A^{-\gamma}$ and $A^{-\gamma}_0$ correspond to the weighted spaces $H^\infty_v$ and $H^0_v$ for the weight function $v_\gamma(r):=(1-r)^\gamma$ on $[0,1)$. The relevant point is that $v_\gamma$ is log-convex, that $\lim_{r\to 1^-}v_\gamma(r)=0$ and that $v_\gamma$ satisfies both \eqref{eq.Con-D} and \eqref{eq.D}. So, for  functions $g\in \cB$ such that $g', \frac{1}{g'}\in H^\infty$, we can apply the previous results. As a sample (cf. Corollary \ref{C.Opt_Kor}), the optimal domain spaces are given by
\begin{equation}\label{eq.15}
	[V_g,A^{-\gamma}]=[T_g,A^{-\gamma}]=A^{-(\gamma+1)},\quad \gamma>0.
	\end{equation}
Or, if $g\in \cB_0$ satisfies both $g', \frac{1}{g'}\in H^\infty$, then $A^{-\gamma}$ is a \textit{proper} subspace of the optimal domain spaces
\begin{equation}\label{eq.16}
	[V_g,A^{-\gamma}_0]=[T_g,A^{-\gamma}_0]=A^{-(\gamma+1)}_0,\quad \gamma>0;
\end{equation}
see Corollary \ref{C.OpK_c}. And so on.

The setting of Subsection 3.2 is the class of Hardy spaces $H^p$, for $1\leq p\leq\infty$, and the disc algebra $A(\D)$. As noted above, for $g\in BMOA$ and $1\leq p<\infty$ the optimal domain spaces $[V_g,H^p]$ are investigated in \cite{BDNS}. This investigation is supplemented in the current paper by allowing $p=\infty$ and incorporates also the related operators $T_g$. The forward shift operator $(Sf)(z):=zf(z)$, for $z\in\D$, which is continuous on $H(\D)$, plays an important role. This operator, when acting in $H^p$, for $1\leq p\leq\infty$, and in $A(\D)$, has operator norm 1 and is injective. The properties of $S$ are used to show, for each $1\leq p\leq\infty$ and $g\in H(\D)$ such that $V_g\colon H^p\to H^p$ is continuous, that the optimal domain spaces $[V_g, H^p]=[T_g,H^p]$ are actually Banach spaces; see Proposition \ref{P.Id_p}. If, in addition, $V_g\colon H^p\to H^p$ is compact, then $H^p$ is a \textit{proper} subspace of $[V_g, H^p]$ (cf. Proposition \ref{P.Hp_inclusion}). Or, if $|g'|>0$ in $\D$ and $T_g\colon H^p\to H^p$ is continuous, then $[T_g,H^p]=[V_g, H^p]$ is a Banach space isometric to $H^p$. If, in addition, $T_g\colon H^p\to H^p$ is not surjective, then also $H^p\varsubsetneqq [T_g,H^p]$ (cf. Proposition \ref{Fact 4}). Concerning $A(\D)$, it is shown that $H^\infty$ is a proper subspace of $[V_g,A(\D)]=[T_g, A(\D)]$ whenever $g\in H(\D)$ has the property that $V_g\colon A(\D)\to A(\D)$ is compact; see Proposition \ref{P.Hinfinito_inclusione}. If $g\in H(\D)$ satisfies $g', \frac{1}{g'}\in H^\infty$ and $V_g\colon A(\D)\to A(\D)$ is continuous, then 
\[
H^1\subseteq [V_g,A(\D)]=[T_g, A(\D)]\subseteq A^{-1}_0\ \mbox{ and }\ H^1\subseteq [V_g,H^\infty]=[T_g, H^\infty]\subseteq A^{-1} ;
\]
see Proposition \ref{P.A_H}. For symbols $g\in H(\D)$ belonging to one of the \textit{Lipschitz spaces} $\lambda_\alpha\subseteq \Lambda_\alpha$, for certain $\alpha>0$, it is shown in Proposition \ref{P.HpHq} that
\[
H^p\subseteq [V_g,H^q]=[T_g,H^q],
\]
with a continuous inclusion, whenever $1\leq p<q$ satisfy $\left(\frac{1}{p}-\frac{1}{q}\right)\leq 1$ and $g\in \Lambda_{\frac{1}{p}-\frac{1}{q}}$, and also that the previous containment is \textit{strict} in the event that $\left(\frac{1}{p}-\frac{1}{q}\right)< 1$ and $g\in \lambda_{\frac{1}{p}-\frac{1}{q}}$.

Section 4 treats the optimal domain spaces of $V_{g_t}$ and $T_{g_t}$ for the particular symbols
\begin{equation}\label{eq.17}
	g_t(z):=-\frac{{\rm Log}(1-tz)}{t},\quad z\in\D,
	\end{equation}
for each $t\in (0,1]$. These symbols have rather pleasant features; see Lemma \ref{P.Propgt}. For instance, each function $g_t\in \cB_0$ and both $g_t', \frac{1}{g'_t}\in H^\infty$. For $t=0$, define $g_0(z):=z$ on $\D$. The operators $C_t:=T_{g_t}$, for $t\in (0,1]$, are known as \textit{generalized Cesàro operators}. Several aspects of these operators (eg., continuity, compactness, spectrum, mean ergodicity) have been studied in various Banach spaces of analytic functions on $\D$ in the recent articles \cite{ABR_AIOT}, \cite{ABR_CM}. Our aim here is to investigate their optimal domain spaces. Given a log-convex weight $v$ on $[0,1)$ satisfying $\lim_{r\to 1^-}v(r)=0$ and with $w(r):=(1-r)v(r)$  on $[0,1)$ it is shown, whenever $D\colon H^\infty_v\to H^\infty_w$ and $J\colon H^\infty_w\to H^\infty_v$ are continuous, that \eqref{eq.14} holds with $g_t$ in place of $g$, that
\[
[V_{g_t},H^0_v]=[C_t,H^0_v]=H^0_w,\quad t\in [0,1),
\]
as Banach spaces and that $H^\infty_v$ is a proper subspace of $[C_t,H^0_v]$; see Proposition \ref{P.C_tHv}. Concerning the Korenblum growth Banach spaces, we mention that \eqref{eq.15} holds with $g_t$ in place of $g$, as does \eqref{eq.16}. In Proposition \ref{P.Kor} it is shown that $A^{-\gamma}$ is a proper subspace of $[C_t, A^{-\gamma}_0]$. For each $ t\in [0,1)$, various properties of the optimal domain space of $C_t$ in the Hardy spaces $H^p$, for $1\leq p<\infty$, (resp. in the disc algebra $A(\D)$) are presented in Proposition \ref{P.Hpp} (resp. Proposition \ref{P.A_Ct}).

In Section 2 of \cite{ABR_JFA} various important properties of the optimal domain of general operators $T\in \cL(H(\D))$ satisfying $T(X)\subseteq X$, with $X$ a Banach space of analytic functions on $\D$, are established. As noted above, further relevant results in this direction are also presented in Section 2. In the final Section 5 we address the special situation when an operator $T\in \cL(H(\D))$ is \textit{not} injective (unlike in Section 2) and has the property that $X:=\Ker(T)$ is \textit{finite dimensional}; see Proposition \ref{P.2} which reveals that some novel features arise in this situation. Examples \ref{Ex1} and \ref{Ex2} illustrate that this can actually occur.

\section{Properties of general  optimal domain spaces}
Given a Banach space $X$, its norm is denoted by $\|\cdot\|_X$, its dual Banach space by $X^*$ and its bidual Banach space by $X^{**}$. The space of all continuous linear operators from $X$ into a Banach space $Y$ is written as $\cL(X,Y)$ or, if $X=Y$, simply as $\cL(X)$. The norm of $T\in \cL(X,Y)$ is given by $\|T\|_{X\to Y}:=\sup_{\|x\|_X\leq 1}\|Tx\|_Y$. The dual operator of $T\in \cL(X,Y)$ is indicated by $T^*\in \cL(Y^*,X^*)$ and its bidual operator by $T^{**}\in \cL(X^{**},Y^{**})$. Finally, the spectrum $\sigma(T;X)$ of $T\in \cL(X)$ is the disjoint union of its point spectrum $\sigma_{pt}(T;X)$, its continuous spectrum $\sigma_{c}(T;X)$ and its residual spectrum $\sigma_{r}(T;X)$. For all these concepts we refer, for example, to \cite{Dun}.

We recall that    $H(\D)$ is a Fr\'echet-Montel space, \cite[\S 27.3(3)]{23}. An increasing  family of norms determining $\tau_c$ is given by
\begin{equation}\label{eq.norme-sup}
	q_r(f):=\sup_{|z|\leq r}|f(z)|,\quad f\in H(\D),
\end{equation}
 for each $r\in (0,1)$. Functions  $f\in H(\D)$  are identified with their Taylor coefficients $\hat{f}:=(\hat{f}(n))_{n\in\N_0} =(\frac{f^{(n)}(0)}{n!})_{n\in\N_0}$,   that is, $f(z)=\sum_{n=0}^\infty \hat{f}(n)z^n$, for $z\in\D$. For each  $z\in\D$ it follows from \eqref{eq.norme-sup} that the evaluation functional $\delta_z\colon  f\mapsto f(z)$, for $f\in H(\D)$, is linear and continuous, that is, $\delta_z\in H(\D)^*$. The space of all continuous linear operators from $H(\D)$ into itself is denoted by $\cL(H(\D))$.

A linear space $X\subseteq H(\D)$ is said to be a \textit{Banach space of analytic functions} on $\D$ if it is a Banach space for some norm $\|\cdot \|_X$ such that the natural inclusion
of $X$ into $H(\D)$ is continuous. The closed graph theorem implies that $T\in \cL(X)$ whenever 
 $T\colon H(\D)\to H(\D)$ is a continuous,  linear operator satisfying $T(X)\subseteq X$.

 The \textit{optimal domain} of an operator  $T\in\cL(H(\D))$ satisfying $T(X)\subseteq X$, with $X$ a Banach space of analytic functions on $\D$, is the linear subspace
 \[
 [T,X]:=\{f\in H(\D)\colon\ Tf\in X\}
 \]
 of $H(\D)$  equipped with the semi-norm
 \[
 \|f\|_{[T,X]}:=\|Tf\|_X,\quad f\in [T,X],
 \]
\cite[Definition 2.2]{ABR_JFA}. For an \textit{injective} operator $T\in \cL(H(\D))$  satisfying $T(X)\subseteq X$, it is known  that $([T,X],\|\cdot\|_{[T,X]})$ is always a normed  space, \cite[Proposition 2.4(ii)]{ABR_JFA}. The map $T_{[T,X]}\colon [T,X]\to X$ defined by $T_{[T,X]}f:=Tf$, for $f\in [T,X]$, is an $X$-valued, linear extension of $T\colon X\to X$. For most operators $T$ of interest it turns out  that $X\subsetneqq [T,X]$, that is, $T_{[T,X]}$ is a genuine extension of $T\in \cL(X)$. But, it  can also happen that $[T,X]=X$, that is, $T\in \cL(X)$ is already  optimally defined; see Remark \ref{R.26}.

The aim of this section is to present  various results concerning the optimal domain of operators acting in general Banach spaces of analytic functions on $\D$. Such results will be needed in subsequent sections.

\begin{prop}\label{Fact 1} Let $X$ be a Banach space of analytic functions on $\mathbb{D}$ and $T\in \cL(H(\D))$ be an  isomorphism satisfying $T(X)\subseteq X$.
\begin{itemize}
	\item[\rm (i)]  $([T,X], \|\cdot \|_{[T,X]})$ is a Banach space isometric to $X$.
	\item[\rm (ii)] If $T\colon X\to X$ is not surjective,  then $X$ is  properly contained  in  $[T,X]$.
\end{itemize}
\end{prop}

\begin{proof} (i) Since $T\colon H(\D)\to H(\D)$ is an isomorphism, we can apply Proposition 2.8 in \cite{ABR_JFA}  to conclude that $([T,X], \|\cdot \|_{[T,X]})$ is  isometric to $X$ via the linear map $T_{[T,X]}\colon [T,X]\to X$. Hence, $([T,X], \|\cdot \|_{[T,X]})$ is a Banach space. 
	
	(ii) The operator $T\colon X\to X$ is necessarily injective but, by assumption, it is not surjective, and so  there exists $g\in X\subseteq H(\D)$ such that $g\not\in T(X)$. Since $T\colon H(\D)\to H(\D)$ is surjective,   there is $f\in H(\D)\setminus X$ satisfying $Tf=g\in X$. Accordingly, $f\in [T,X]$ but, $f\not \in X$.
\end{proof}



Given Banach spaces of analytic functions on $\D$, say $X$ and $Y$, such that $X\subseteq Y$ continuously and an operator $T\in \cL(H(\D))$ satisfying $T(Y)\subseteq X$, it follows that necessarily 
  $T\in \cL(Y,X)$. Indeed, let $(f_n)_{n\in\N}\subseteq Y$ satisfy $f_n\to 0$ in $Y$ and $Tf_n\to g$ in $X$ for some $g\in X$, as $n\to\infty$. Then also $f_n\to 0$ and $Tf_n\to g$ in $H(\D)$ as $n\to\infty$ because both  $X$ and $Y$ are Banach spaces of analytic functions on $\D$. Since $T\in \cL(H(\D))$, it follows that also $Tf_n\to 0$ in $H(\D)$ as $n\to\infty$. Therefore, $g=0$. By the closed graph theorem we can conclude that $T\in \cL(Y,X)$. Note that also    $Y\subseteq [T,X]$ with a continuous inclusion. To see this let $h\in Y$. Then $h\in H(\D)$ and $Th\in T(Y)\subseteq X$. Accordingly, $h\in [T,X]$. So, we can consider the natural inclusion map $\mathscr{J}_{[T,X]}\colon Y\to [T,X]$ defined by $\mathscr{J}_{[T,X]}f:=f$ for $f\in Y$. Since
\[
\|\mathscr{J}_{[T,X]}f\|_{[T,X]}=\|f\|_{[T,X]}=\|Tf\|_X\leq \|T\|_{Y\to X}\|f\|_Y, \ f\in Y,
\] 
it follows that  $\mathscr{J}_{[T,X]}\colon Y\to [T,X]$ is continuous.

The following result is   a refinement of \cite[Propositiom 2.9]{ABR_JFA}. 

\begin{prop}\label{P5} Let $X$ and $Y$ be  Banach spaces of analytic functions on $\D$ satisfying $X\subseteq Y$ with a continuous inclusion. Let $T\in\cL(H(\D))$ be an injective operator such that $([T,X],\|\cdot\|_{[T,X]})$ is a Banach space and $T(Y)\subseteq X$ (hence, also $T(X)\subseteq X$).  The following properties are equivalent.
	\begin{itemize}
		\item[\rm (i)] The continuous linear operator $T\colon Y\to X$ has closed range in $X$.
		\item[\rm (ii)]  The natural inclusion map $\mathscr{J}_{[T,X]}\colon Y\to [T,X]$, defined by $f\mapsto \mathscr{J}_{[T,X]}f:=f$,  has closed range in $[T,X]$.
	\end{itemize}
\end{prop}

\begin{proof} (i)$\Rightarrow$(ii) 
	Since $T(Y)\subseteq X$ is a closed subspace of $X$, necessarily $(T(Y),\|\cdot\|_X)$ is a Banach space. Moreover, $T\in \cL(Y,X)$ implies that $T\colon Y\to (T(Y), \|\cdot\|_X)$ is continuous.
	So,   the open mapping theorem implies that the continuous, bijective operator $T\colon Y\to T(Y)$ is an isomorphism. Accordingly,   there exists $c>0$ such that
	\begin{equation*}\label{eq.closed}
		\|Tf\|_X\geq c\|f\|_Y, \quad f\in Y.
		\end{equation*}
	It follows from the previous inequality that
	\[
	\|\mathscr{J}_{[T,X]}f\|_{[T,X]}=\|f\|_{[T,X]}=\|Tf\|_X\geq c\|f\|_Y, \quad f\in Y.
	\]
	Since  $\mathscr{J}_{[T,X]}$ is continuous (as observed prior to the proposition) we can conclude that the inclusion map $\mathscr{J}_{[T,X]}\colon Y\to [T,X]$ has closed range.
	
	(ii)$\Rightarrow$(i) The inclusion map $\mathscr{J}_{[T,X]}\colon Y\to [T,X]$ has closed range and hence, the linear subspace $\mathscr{J}_{[T,X]}(Y)=Y$ is  closed in the Banach space $[T,X]$. So, $(Y, \|\cdot\|_{[T,X]})$ is a Banach space. In view of the continuity of $\mathscr{J}_{[T,X]}\colon (Y,\|\cdot\|_Y)\to ([T,X], \|\cdot\|_{[T,X]})$,  the identity operator $L\colon (Y,\|\cdot\|_Y)\to (Y,\|\cdot\|_{[T,X]})$ is continuous and bijective. We can apply the open mapping theorem to conclude that  $L\colon (Y,\|\cdot\|_Y)\to (Y,\|\cdot\|_{[T,X]})$ is an isomorphism. Accordingly, there exists $c>0$ satisfying $\|Lf\|_{[T,X]}\geq c\|f\|_Y$ for every $f\in Y$. Hence,
	\[
	\|Tf\|_X=\|f\|_{[T,X]}=\|Lf\|_{[T,X]}\geq c\|f\|_Y,\quad f\in Y,
	\]
	which implies that the operator $T\colon Y\to X$ has closed range.
	\end{proof}

An argument as in the proof of Lemma 5.8 in \cite{ABR_JFA} establishes the following fact.

\begin{lemma}\label{Lemma_Compact} Let $(X,\|\cdot\|_X)$ and $(Y,\|\cdot\|_Y)$ be  infinite dimensional Banach spaces and $T\in \cL(X,Y)$ be injective and compact. Then $T$ does not have a closed range in $Y$.
\end{lemma}

An immediate consequence of Proposition \ref{P5} is the following fact.

\begin{corollary}\label{C1}
Let $X$ and $Y$ be  Banach spaces of analytic functions on $\D$ satisfying $X\subseteq Y$ with a continuous inclusion. Let  $T\in\cL(H(\D))$ be an injective operator such that $([T,X],\|\cdot\|_{[T,X]})$ is a Banach space and  $T(Y)\subseteq X$. 
\begin{itemize}
		\item[\rm (i)] Suppose that the natural inclusion map $\mathscr{J}_{[T,X]}\colon Y\to [T,X]$ is surjective. Then  the operator $T\colon Y\to X$ has closed range in $X$.
	\item[\rm (ii)] Suppose that the operator $T\colon Y\to X$ fails to have closed range in $X$. Then $Y$ is a proper subspace of $[T,X]$.
	\item[(iii)] Suppose that $X$ and $Y$ are infinite-dimensional and $T\in\cL(Y,X)$ is a compact operator. Then $Y$ is a proper subspace of $[T,X]$.
\end{itemize}
\end{corollary}

\begin{proof} 
(i) We have seen that
the inclusion map $\mathscr{J}_{[T,X]}\colon Y\to [T,X]$ is a continuous bijection of $Y$ onto its range in $[T,X]$. Accordingly, it is an isomorphism by the open mapping theorem. So, $\mathscr{J}_{[T,X]}$ has a closed range in $[T,X]$. Then Proposition \ref{P5} implies that $T\colon Y\to X$ has a closed range in $X$.

(ii) Suppose it was the case that $Y=[T,X]$. Then the inclusion map $\mathscr{J}_{[T,X]}\colon Y\to [T,X]$ is surjective. Since $([T,X],\|\cdot\|_{ [T,X]})$ is a Banach space,  again the open mapping theorem implies that  $\mathscr{J}_{[T,X]}\colon Y\to [T,X]$ is an isomorphism and hence, it has a closed range in $[T,X]$. Then Proposition  \ref{P5} yields that the operator $T\colon Y\to X$ has  closed range in $X$; a contradiction. So, $Y\subsetneqq [T,X]$. 

(iii) Since $T\in\cL(Y,X)$ is compact, the operator $T\colon Y\to X$ fails to have closed range in $X$; see Lemma \ref{Lemma_Compact}. So, the result follows from part (ii).
\end{proof}

Let $X$ be a Banach space of analytic functions on $\D$ and $T\in \cL(H(\D))$ satisfy $T(X)\subseteq X$, in which case $T\in \cL(X)$. Then $X\subseteq [T,X]$ with a continuous inclusion and the linear extension $T_{[T,X]}\colon [T,X]\to X$ of $T\colon X\to X$ is continuous, \cite[Proposition 2.4 (iii), (iv)]{ABR_JFA}. In general, $X$ need \textit{not} be dense in $[T,X]$. Indeed, let $X:=H^2$ and $T=V_g$ with $g\in BMOA$ given by $g(z):=\int_0^z\exp(\frac{\xi+1}{\xi-1})\,d\xi$, for $z\in\D$. Then the polynomials are not dense in $[V_g,H^2]$, \cite[Proposition 4]{BDNS}, and hence, $H^2$ is not dense in $[V_g,H^2]$, \cite[Proposition 3]{BDNS}. For a further example, see Remark \ref{R.38}. The following result is an extension of \cite[Proposition 3]{BDNS}.

\begin{prop}\label{P.Denso} Let $X$ be a Banach space of analytic functions on $\D$, let $Y\subseteq X$ be a dense linear subspace of $X$ and let $T\in \cL(H(\D))$ satisfy $T(X)\subseteq X$.  Then $Y$ is dense in $[T,X]$ if and only if $X$ is dense in $[T,X]$.
\end{prop}

\begin{proof} Suppose that $X$ is dense in $[T,X]$. Given $\varepsilon >0$ and $f\in [T,X]$ there exists $h\in X$ such that $\|f-h\|_{[T,X]}<\varepsilon$. There also exists $g\in Y$ such that $\|h-g\|_X<\varepsilon$. Since $T\in \cL(X)$, we have
	\[
	\|h-g\|_{[T,X]}=\|T(h-g)\|_X\leq \|T\|_{X\to X}\|h-g\|_X
	\]
	from which it follows that 
	\[
	\|f-g\|_{[T,X]}\leq \|f-h\|_{[T,X]}+\|h-g\|_{[T,X]}\leq (1+\|T\|_{X\to X})\varepsilon.
	\]
	This implies that $Y$ is dense in $[T,X]$.
	
	Conversely, if $Y$ is dense in $[T,X]$, then so is $X$ as $Y\subseteq X$.
\end{proof}

\begin{remark}\label{R.26}\rm 
The examples of operators $T\in \cL(H(\D))$ and spaces $X$ exhibited in this paper typically satisfy $X\subsetneqq [T,X]$. Hovewer, it can also happen that $X=[T,X]$. 
For instance, if $h\in H^\infty$, then the \textit{multiplication operator} $T_h\colon H(\D)\to H(\D)$ given by $f\mapsto hf$ is continuous and, whenever $|h|>0$ on $\D$, it  is also injective. Let $X:=H^p$ with $1\leq p\leq \infty$, in which case it is clear that $T_h(X)\subseteq X$. In the event that also $\frac{1}{h}\in H^\infty$ it is routine to check that $[T_h,H^p]=H^p$. On the other hand, if $h(z):=z-u$ on $\D$, say, for some $u$ with $|u|=1$, then $T_h(\frac{1}{h})=\mathbbm{1}\in H^\infty$ but, $\frac{1}{h}\in H(\D)\setminus H^\infty$. Hence, $H^\infty\subsetneqq [T,H^\infty]$. A similar phenomenon occurs for the weighted  Banach spaces $H^\infty_v$ (which arise in Section 3), that is, if $h, \frac{1}{h}\in H^\infty$, then $[T_h,H^\infty_v]=H^\infty_v$. For the case when $T\in \cL(H(\D))$ is a \textit{composition operator} $C_\varphi\colon f\mapsto f\circ \varphi$, for $f\in H(\D)$, where $\varphi\in {\rm Aut}(\D)$  is any automorphism on $\D$, and $X:=H^\infty_v$ for suitable weights $v$ (which includes certain Korenblum growth Banach spaces), it turns out that $[C_\varphi, H^\infty_v]=H^\infty_v$, \cite[Proposition 2.3]{BDLT}. Note that $C_\varphi\in \cL(H(\D))$ is an isomorphism for every  $\varphi\in  {\rm Aut}(\D)$.
\end{remark}

\section{Optimal domain of Volterra operators in classes of  Banach spaces of analytic functions }

Whenever $g\in H(\D)$ is a non-constant function, recall that the \textit{Volterra operator} $V_g\colon H(\D)\to H(\D)$ is the linear operator defined by the right side of \eqref{eq.GVO}, that is, 
\begin{equation*}\label{eq.V}
	(V_gf)(z):=\int_0^z f(\xi)g'(\xi)\,d\xi,\quad f\in H(\D), \ z\in\D.
	\end{equation*}
It acts continuously in $H(\D)$ and  is \textit{injective} on $H(\D)$, \cite{AND}, but, it is  not surjective because $(V_gf)(0)=0$ for every $f\in H(\D)$.
%
%
 The  operators $V_g$ have been investigated on different spaces of analytic functions by many
authors. We refer to \cite{Cont-a,Si}, for example,  and the references therein. For $g(z):=z$ the operator $V_g$ reduces to the classical integration operator, also called the Volterra operator.

Related to the  operator $V_g$, with $g\in H(\D)$, is the operator $T_g\colon H(\D)\to H(\D)$ defined by \eqref{eq.CC}.
 Clearly, $T_g\in \cL(H(\D))$. It follows from the injectivity of $V_g$ that $T_g\in \cL(H(\D))$ is also \textit{injective}; see \cite[Section 5]{ABR_JFA}. Under suitable assumptions on $g$, the operator  $T_g\in \cL(H(\D))$ is also surjective.

\begin{prop}\label{Fact 2} Let $g\in H(\D)$ satisfy $|g'|>0$ in $\D$. Then $T_g\colon H(\D)\to H(\D)$ is an isomorphism.
\end{prop}

\begin{proof} We  know that $T_g\colon H(\D)\to H(\D)$ is  continuous and injective. So, it remains to show that  $T_g\colon H(\D)\to H(\D)$ is also surjective.
	
	Fix $h\in H(\D)$ and consider the equation $T_gf=h$, where $f$ is to be determined. From the definition of  $T_g$ it follows that 
	$	\int_0^zf(\xi)g'(\xi) d\xi)=zh(z)$, for $z\in\D$.
	Differentiating both sides of the previous equation yields
	$	f(z)g'(z)=h(z)+zh'(z)$, for $z\in \D$, that is, 
	\[
	f(z)=\frac{h(z)+zh'(z)}{g'(z)},\quad z\in \D.
	\]
	Since  $|g'|>0$ in $\D$, the function $f$ defined as above belongs to $H(\D)$  and  it solves the equation $T_gf=h$. Accordingly, $T_g\colon H(\D)\to H(\D)$ is  surjective with inverse $T_g^{-1}\colon H(\D)\to H(\D)$ defined by
	\[
	(T^{-1}_gh)(z):=\frac{h(z)+zh'(z)}{g'(z)},\quad z\in \D.
	\]
	By the open mapping theorem for Fr\'echet spaces $T_g^{-1}\in \cL(H(\D))$.
\end{proof}

\subsection{Weighted Banach spaces of analytic functions } 
Let $H^\infty$ denote the space of all bounded, analytic functions on $\D$, which
 is a Banach space relative to the norm
\begin{equation}\label{eq.supnorm}
	\|f\|_\infty:=\sup_{z\in\D}|f(z)|,\quad f\in H^\infty.
\end{equation}
The disc algebra $A(\D)$ is the linear subspace of $H(\D)$  consisting  of all functions which extend to a continuous function on the closure $\overline{\D}$ of $\D$. That is, $h\in A(\D)$ if and only if there exists a (unique) function $\tilde{h}\in C(\overline{\D})$ satisfying  $h(z)=\tilde{h}(z)$ for all $z\in\D$. In particular, $h\in H^\infty$.
The space $A(\D)$ is   a proper, closed subalgebra of the abelian Banach algebra $H^\infty$.

A \textit{weight} $v$ is a continuous, non-increasing function
$v\colon [0,1)\to (0,\infty)$, which is  extended  to $\D$ by setting $v(z):=v(|z|)$, for $z\in\D$. Note that $v(z)\leq v(0)$ for all $z\in\D$. Given a weight $v$ on $[0,1)$, two associated   \textit{weighted Banach spaces of analytic functions} on $\D$ are defined  by
\[
H_v^\infty:=\{f\in H(\D)\, :\, \|f\|_{\infty,v}:=\sup_{z\in\D}|f(z)|v(z)<\infty\},
\]
and
\[
H^0_v:=\{f\in H(\D)\, :\, \lim_{|z|\to 1^-}|f(z)|v(z)=0\},
\]
both equipped with the norm $\|\cdot\|_{\infty,v}$. Noting that $\|f\|_{\infty,v}\leq v(0)\|f\|_\infty$ for all $f\in H^\infty$, it is clear that $H^\infty\su H^\infty_v$ with a continuous inclusion. Fix $r\in (0,1)$. Given $f\in H^\infty_v$ there exists $z_0\in D$ with $|z_0|=r$ such that $\sup_{|z|\leq r}|f(z)|=|f(z_0)|$ and hence, \eqref{eq.norme-sup} implies that
\[
q_r(f)=|f(z_0)|=\frac{1}{v(r)}|f(z_0)|v(r)\leq \frac{1}{v(r)}\|f\|_{\infty,v}.
\]
It follows that $H^\infty_v\subseteq H(\D)$ continuously.  Point evaluations on $\D$ belong to both $(H^\infty_v)^*$ and $(H^0_v)^*$, \cite[Lemma 1]{Sh-Wi}.
  Moreover, if $v$ satisfies  $\lim_{r\to 1^-}v(r)=0$, then the space  $H^0_v\not=\{0\}$ coincides with the closure in $H^\infty_v$ of the space  $\cP$ of all polynomials, \cite[Lemma 3]{Sh-Wi}. 
For the constant function $v(z)=1$ on $\D$, the space $H^\infty_v$  is precisely the space $H^\infty$ and $H^0_v$ reduces to $\{0\}$.

We refer  to \cite{Bo} for a recent survey of such types of weighted Banach spaces and operators between them. 

The \textit{differentiation operator} $D\colon H(\D)\to H(\D)$ is given by $f\mapsto Df:=f'$. It is routine to verify that $D\in \cL(H(\D)))$. Given weights $v$ and $w$ on $[0,1)$, both of the Banach spaces $H^\infty_v$ and $H^\infty_w$ are continuously included in $H(\D)$. Whenever $D$ satisfies $D(H^\infty_v)\subseteq H^\infty_w$ we say that $D$ exists and indicate this simply by writing $D\colon H^\infty_v\to H^\infty_w$. An application of the closed graph theorem implies that necessarily  $D\in \cL(H^\infty_v,H^\infty_w)$. A weight function $v$ on $[0,1)$ is called \textit{log-convex} if the function $x\mapsto -\log v(e^x)$, for $x\in (-\infty,0)$, is convex; see \cite[p.1145]{AT}, \cite[pp.156-157]{BBT}, \cite{Bo} and the references therein. For such weights $v$ the conditions (i)-(vi) of \cite[Lemma 2.6]{AT} are all equivalent (warning: in \cite{AT} the weight function $v$ corresponds to $1/v$ in this paper) and hence, via \cite[Theorem 2.8]{AT}, are equivalent to the continuity (i.e., existence) of $D\colon H^\infty_v\to H^\infty_w$, where $w(r):=(1-r)v(r)$, for $r\in [0,1)$. Condition (ii) of Lemma 2.6 in \cite{AT},  for $v$,  is the requirement that 
\begin{equation}\label{eq.Con-D} 
	v(r)(1-r)^{-\alpha}\ \textrm{is increasing on $[r_0,1)$ for some $\alpha>0$ and some $r_0\in (0,1)$};
	\end{equation}
see also \cite[Theorem 27]{Bo}.

The continuity and compactness of the Volterra operator $V_g$ acting  on
$H^\infty_v$ and $H^0_v$ was investigated in \cite{BCHMP}. The following few results concern the optimal domain of such operators.

\begin{prop}\label{P_H_v} Let $g\in H(\D)$ be a non-constant function,  $v\colon [0,1)\to (0,\infty)$ be a log-convex weight function satisfying \eqref{eq.Con-D} and $w(r):=(1-r)v(r)$, for $r\in [0,1)$.  
	\begin{itemize}
		\item[\rm (i)]  Suppose that the Volterra  operator $V_g\colon H^\infty_v\to H^\infty_v$ is continuous. Then $([V_g,H^\infty_v], \|\cdot\|_{[V_g,H^\infty_v]})$ is a Banach space. 
		\item[\rm (ii)] Suppose, in addition, that the weight function $v$ satisfies  the condition $\lim_{r\to 1^-}v(r)=0$ and  the Volterra operator $V_g\colon H^0_v\to H^0_v$ is continuous. Then  $([V_g,H^0_v], \|\cdot\|_{[V_g,H^0_v]})$ is a Banach space. 
	\end{itemize}
\end{prop}

\begin{proof} (i) By Proposition 2.7 of \cite{ABR_JFA} to establish that $[V_g,H^\infty_v]$ is a Banach space,  it suffices to show that the natural inclusion map $\mathscr{J}\colon [V_g, H^\infty_v]\to H(\D)$, defined by  $f\mapsto \mathscr{J}f:=f$, is continuous. To achieve this we proceed as in the proof of \cite[Proposition 3.3]{ABR_JFA}. By continuity of $D\colon H^\infty_v\to H^\infty_w$ (see the discussion prior to this proposition),  there exists $A>0$ such that
	\begin{equation}\label{eq.Der-1}
		\|h'\|_{\infty,w}=\|Dh\|_{\infty,w}\leq A\|h\|_{\infty,v},\quad h\in H^\infty_v.
	\end{equation}
	Fix any $r\in (0,1)$. Now, $g'$ is not identically zero on $\D$ and the zeros of $g'\in H(\D)$ are isolated. Hence,  there is $t\in (r,1)$ for which $g'(z)\not=0$, for every $z\in \D$ satisfying $|z|=t$ and so $M:=\min_{|z|=t}|g'(z)|>0$. Given $f\in [V_g,H^\infty_v]$, 
	it follows, for all $u\in\D$ satisfying $|u|\leq r$, that
	\begin{equation}\label{eq,st-1}
		|f(u)|\leq \max_{|z|=t}|f(z)|=\max_{|z|=t}\frac{|f(z)g'(z)|}{|g'(z)|}\leq \frac{1}{M}\max_{|z|=t}|f(z)g'(z)|.
	\end{equation}
	Since
	$f(z)g'(z)=(V_gf)'(z)$, for $z\in\D$,
	the inequality \eqref{eq.Der-1} implies that
	\begin{equation}\label{eq.Deri-1}
		\|fg'\|_{\infty,w}=	\|(V_gf)'\|_{\infty,w}\leq A\|V_gf\|_{\infty,v}=A\|f\|_{[V_g, H^\infty_v]}.
	\end{equation}
	Combining \eqref{eq,st-1} and \eqref{eq.Deri-1} we can conclude, for all $u\in\D$ with $|u|\leq r$, that
	\begin{align*}
		|f(u)|\leq & \frac{1}{M}\max_{|z|=t}|(V_gf)'(z)|
		=\frac{1}{M}\frac{1}{w(t)}\max_{|z|=t}w(z)|(V_gf)'(z)|\\
		\leq & \frac{1}{M}\frac{1}{w(t)}\sup_{z\in\D}w(z)|(V_gf)'(z)|
		\leq  \frac{A}{Mw(t)}\sup_{z\in\D}v(z)|(V_gf)(z)|\\
		=& \frac{A}{Mw(t)}\|V_gf\|_{\infty,v}=\frac{A}{Mw(t)}\|f\|_{[V_g,H^\infty_v]}.
	\end{align*}
	According to \eqref{eq.norme-sup} this implies that
	\[
	q_r(f)=\sup_{|u|\leq r}|f(u)|\leq \frac{A}{Mw(t)}\|f\|_{[V_g,H^\infty_v]}.
	\]
	But, $r\in (0,1)$ and $f\in [V_g, H^\infty_v]$ are arbitrary and so  the inclusion map $\mathscr{J}\colon [V_g, H^\infty_v]\to H(\D)$ is continuous, as required.

	(ii) We first note that $D\colon H^0_v\to H^0_w$ is also continuous. This  follows from the continuity of $D\colon H^\infty_v\to H^\infty_w$, the fact noted above that the space $\cP$ of all polynomials is dense in $H^0_v$ and that $D(\cP)\subseteq \cP$. The proof then proceeds as in  part (i).
\end{proof}

\begin{remark}\label{R.Hinfinito}\rm
	(i) A log-convex weight function $v$ need \textit{not} satisfy $\lim_{r\to 1^-}v(r)=0$ (assumed in part (ii) in Proposition \ref{P_H_v}); consider $v(r):=e^{-r}$, for $r\in [0,1)$. 
	
	(ii) Let $v\colon [0,1)\to (0,\infty)$ be any log-convex weight function satisfying \eqref{eq.Con-D} and $w(r):=(1-r)v(r)$, for $r\in [0,1)$. Since  $D\colon H^\infty_v\to H^\infty_w$ is continuous and $H^\infty\subseteq H^\infty_v$ with a continuous inclusion, the restricted differentiation operator $D\colon H^\infty\to H^\infty_w$ is also continuous. So,  in the proof of Proposition \ref{P_H_v}(i) we can replace $\|\cdot\|_{\infty,v}$ throughout  with $\|\cdot\|_\infty$ to obtain, whenever  $V_g\colon H^\infty\to H^\infty$ is continuous,  that  $([V_g, H^\infty],\|\cdot\|_{[V_g,H^\infty]})$ is a Banach space.
\end{remark}

\begin{remark}\label{R.ContV_g_H}\rm  Let $g\in H(\D)$ be a non-constant function and $v\colon [0,1)\to (0,\infty)$ be a weight function satisfying $\lim_{r\to 1^-}v(r)=0$ for which the operator $V_g\colon H^0_v\to H^0_v$ is continuous. Then $(V_g)^{**}=V_g$, \cite[Lemma 1]{BCHMP}, and $(H^0_v)^{**}=H^\infty_v$, \cite[Example 2.1]{BS}, imply that $V_g\colon H^\infty_v\to H^\infty_v$ is also continuous. So, under the assumptions of Proposition \ref{P_H_v}, the proof of part (ii) follows by observing that $([V_g,H^0_v],\|\cdot\|_{[V_g,H^0_v]})$ is a closed subspace of $([V_g,H^\infty_v],\|\cdot\|_{[V_g,H^\infty_v]})$ because $H^0_v$ is a closed, $V_g$-invariant subspace of $H^\infty_v$, \cite[Proposition 2.4(v)]{ABR_JFA}. 
\end{remark}

\begin{prop}\label{P.H_vcontained} Let $g\in H(\D)$ be a non-constant function and $v\colon [0,1)\to (0,\infty)$ be a log-convex weight function satisfying both condition \eqref{eq.Con-D} and $\lim_{r\to 1^-}v(r)=0$. Whenever the operator $V_g\colon H^0_v\to H^0_v$ is compact,  $H^\infty_v$ is a proper subspace of $[V_g,H^0_v]$.
\end{prop}

\begin{proof} Since $V_g\colon H^0_v\to H^0_v$ is compact (by assumption) and  $(V_g)^{**}=V_g\colon H^\infty_v\to H^\infty_v$ (cf. Remark \ref{R.ContV_g_H}),  via \cite[\S 42.2(1) and (2)]{24} we can conclude that $V_g(H^\infty_v)\subseteq H^0_v$. Accordingly, $V_g\colon H^\infty_v\to H^0_v$ is also compact.
	On the other hand, the assumptions on $v$ ensure that   $D\colon H^\infty_v\to H^\infty_w$  is continuous, where $w(r):=(1-r)v(r)$, for $r\in [0,1)$, and so 
	  Proposition \ref{P_H_v}(ii) implies  that $([V_g,H^0_v],\|\cdot\|_{[V_g,H^0_v]})$ is a Banach space. So, by Corollary \ref{C1}(iii) we can conclude that $H^\infty_v$ is a proper subspace of $[V_g, H^0_v]$.
\end{proof}

The \textit{integration operator} $J\colon H(\D)\to H(\D)$ is defined by $f\mapsto (Jf)(z):=\int_0^zf(\xi)d\xi$. It is routine to verify that $J\in \cL(H(\D))$. Given weights $v$ and $w$ on $[0,1)$ we write $J\colon H^\infty_w\to H^\infty_v$ whenever $J(H^\infty_w)\subseteq H^\infty_v$ and say that $J$ exists. By the closed graph theorem this is equivalent to $J\in \cL(H^\infty_w,H^\infty_v)$. We point out that
\begin{equation}\label{eq.BB}
	D(Jh)=h, \quad \forall h\in H(\D),
\end{equation}
and, for every $g\in H(\D)$, that
\begin{equation}\label{eq.C}
	V_gf=J(fg'),\quad \forall f\in H(\D).
	\end{equation}
Corollary 3.2 in \cite{AT} gives sufficient conditions on $v$ and $w$ to ensure that $J\colon H^\infty_w\to H^\infty_v$ is continuous. Let $v$ be a log-convex weight and $w(r):=(1-r)v(r)$ on $[0,1)$. The  condition
\begin{equation}\label{eq.D}
	v(r)(1-r)^{-\alpha}\ \textrm{is decreasing on $[r_0,1)$ for some $\alpha>0$ and some $r_0\in (0,1)$}
	\end{equation}
is equivalent to the continuity of $J\colon H^\infty_w\to H^\infty_v$, \cite[Lemma 3.15 \& Theorem 3.16]{AT}; see also Theorem 30 in \cite{Bo}, where the function $w$ should be $w(r):=(1-r)v(r)$ on $r\in [0,1)$ and not as given there.

\begin{prop}\label{P.Opt_H_v}  Let $g\in H(\D)$ satisfy $g', \frac{1}{g'}\in H^\infty$ and $v\colon [0,1)\to (0,\infty)$ be a log-convex weight function satisfying both \eqref{eq.Con-D} and \eqref{eq.D}.  
	\begin{itemize}
		\item[\rm (i)] Let $w(r):=(1-r)v(r)$, for $r\in [0,1)$. Then both  operators $D\colon H^\infty_v\to H^\infty_w$ and $J\colon H^\infty_w\to H^\infty_v$ are continuous.
		\item[\rm (ii)] Whenever the operator $V_g\colon H^\infty_v\to H^\infty_v$ is continuous,  the optimal domain $[V_g,H^\infty_v]=H^\infty_w$ as linear spaces in $H(\D)$ and with equivalent norms. 
		\item[\rm (iii)] Suppose, in addition, that $\lim_{r\to 1^-}v(r)=0$ and the operator $V_g\colon H^0_v\to H^0_v$ is continuous. Then the optimal domain $[V_g,H^0_v]=H^0_w$ as linear spaces in $H(\D)$ and with equivalent norms. 
	\end{itemize}
\end{prop}

\begin{proof} (i) It has already been noted that continuity of $D\colon H^\infty_v\to H^\infty_w$ (resp. of $J\colon H^\infty_w\to H^\infty_v$) is equivalent to \eqref{eq.Con-D} (resp. \eqref{eq.D}).
	
	(ii) Fix $h\in H^\infty_w$. Then $hg'\in H^\infty_w$ (because $g'\in H^\infty$) and hence, by \eqref{eq.C},  $V_gh=J(hg')\in H^\infty_v$. Accordingly, $h\in [V_g,H^\infty_v]$. Moreover,
	\begin{align}\label{eq.3}
		\|h\|_{[V_g,H^\infty_v]}&=\|V_gh\|_{\infty,v}=\|J(hg')\|_{\infty,v}\leq \|J\|_{H^\infty_w\to H^\infty_v}\|hg'\|_{\infty,w}\nonumber\\
		&\leq \|J\|_{H^\infty_w\to H^\infty_v}\|g'\|_{\infty}\|h\|_{\infty,w}.
	\end{align}
	It follows that $H^\infty_w\subseteq [V_g,H^\infty_v]$ with a continuous inclusion.
	
	Conversely, let $f\in [V_g,H^\infty_v]$, that is,  $f\in H(\D)$ and $V_gf=J(fg')\in H^\infty_v$. According to \eqref{eq.BB} we have $fg'=D(J(fg'))$ and hence, $fg'\in H^\infty_w$ as $J(fg')\in H^\infty_v$ and $D\colon H^\infty_v\to H^\infty_w$. Since $\frac{1}{g'}\in H^\infty$, it follows that also $f=\frac{1}{g'}\cdot (fg')\in H^\infty_w$. Therefore, $ [V_g,H^\infty_v]\subseteq H^\infty_w$. So, $ H^\infty_w=[V_g,H^\infty_v]$ as linear spaces.
	
	Since both $([V_g, H^\infty], \|\cdot\|_{[V_g, H^\infty_v]})$ and $H^\infty_w$ are Banach spaces (cf. Proposition \ref{P_H_v}(i)) and the inclusion $H^\infty_w\subseteq [V_g,H^\infty_v]$ is continuous and surjective,  the open mapping theorem implies that $H^\infty_w=[V_g,H^\infty_v]$ have equivalent norms.
	
	(iii) The result follows as in part (i), after observing that both of the  operators $D\colon H^0_v\to H^0_w$ and $J\colon H^0_w\to H^0_v$ are continuous. Indeed, for the continuity of $D\colon H^0_v\to H^0_w$ see the proof of Proposition \ref{P_H_v}(ii). The continuity of $J\colon H^0_w\to H^0_v$ is a consequence of  the facts that the integration operator $J\colon H^\infty_w\to H^\infty_v$ is  continuous, that the space $\cP$ of all polynomials is dense in $H^0_w$ and that $J(\cP)\subseteq \cP$.
\end{proof}

\begin{corollary}\label{C.Opt_H_v}   Let $v\colon [0,1)\to (0,\infty)$ be a log-convex weight function satisfying both \eqref{eq.Con-D} and \eqref{eq.D},  $w(r):=(1-r)v(r)$, for $r\in [0,1)$, and $g\in H(\D)$ satisfy $g', \frac{1}{g'}\in H^\infty$.
	\begin{itemize}
		\item[\rm (i)]  Whenever the operator $V_g\colon H^\infty_v\to H^\infty_v$ is continuous, the optimal domains $[V_g,H^\infty_v]=[T_g,H^\infty_v]=H^\infty_w$ as linear spaces in $H(\D)$ and with equivalent norms. 
		\item[\rm (ii)] Suppose that $\lim_{r\to 1^-}v(r)=0$ and the operator $V_g\colon H^0_v\to H^0_v$ is continuous. Then the optimal domains $[V_g,H^0_v]=[T_g,H^0_v]=H^0_w$ as linear spaces in $H(\D)$ and with equivalent norms. 
		\item[\rm (iii)] Suppose, in addition, that $\lim_{r\to 1^-}v(r)=0$ and the operator $V_g\colon H^0_v\to H^0_v$ is compact. Then $H^\infty_v$ is a proper subspace of  $[V_g,H^0_v]=[T_g,H^0_v]=H^0_w$
	\end{itemize}
\end{corollary}

\begin{proof}
Part	(i) follows by applying  \cite[Proposition 5.2(i)-(ii)]{ABR_JFA} and Proposition \ref{P.Opt_H_v}(i).
	
Part	(ii) follows via \cite[Proposition 5.2(iii)-(iv)]{ABR_JFA} and Proposition \ref{P.Opt_H_v}(ii).
	
Part	(iii) follows by applying  Proposition \ref{P.H_vcontained} and part (ii).
\end{proof}


For each $\gamma>0$, the spaces $H^\infty_v$ and $H^0_v$, for the log-convex weight function $v_\gamma(r) := (1-r)^\gamma$, for $r\in [0, 1)$, are denoted by $A^{-\gamma}$ and $A^{-\gamma}_0$, respectively and, as already mentioned, are known as the \textit{Korenblum growth Banach spaces (of order $\gamma$)}.

Related to $A^{-\gamma}$ and $A^{-\gamma}_0$ are certain \textit{Bloch spaces}. A function $h\in H(\D)$  belongs to the \textit{Bloch space} $\cB$ whenever
$h'\in  A^{-1}$, that is, $\sup_{z\in\D}(1-|z|)|h'(z)|<\infty$, in which case  $\cB$ is a  Banach space of analytic  functions on $\D$ relative to the norm
\begin{equation}\label{eq.normbloch}
	\|h\|_{\cB}:=|h(0)|+\sup_{z\in\D}(1-|z|)|h'(z)|.
\end{equation}
A function $h\in\cB$ belongs to the \textit{little Bloch space} $\cB_0$ if it satisfies
\begin{equation}\label{eq.lim}
	\lim_{|z|\to 1^-}(1-|z|)|h'(z)|=0.
\end{equation}
It is known that $\cB_0$  is a closed subspace of  $\cB$ and  hence, $\cB_0$  is a  Banach space when it is equipped with the norm  \eqref{eq.normbloch}. Moreover, the bidual $\cB_0^{**}=\cB$. For various properties of Bloch spaces we refer to \cite{ACP}, \cite{Z}, for example.

It was noted in Section 1 that 
a detailed study of the optimal domain spaces $[V_g, A^{-\gamma}]$ and $[V_g, A^{-\gamma}_0]$ is undertaken in \cite[Section 3]{ABR_JFA}.
For $g\in \cB$ a non-constant function and for every $\gamma>0$,  both $[V_g,A^{-\gamma}]$ and $[V_g, A^{-\gamma}_0]$ are Banach spaces, \cite[Proposition 3.3]{ABR_JFA}.

For each $\gamma>0$ and with $v_\gamma(r) := (1-r)^\gamma$ on $[0, 1)$ we note that $v_\gamma(r)(1-r)^{-\frac{\gamma}{2}}=(1-r)^{\frac{\gamma}{2}}$ is increasing on $[0,1)$ and $v_\gamma(r)(1-r)^{-2\gamma}=(1-r)^{-\gamma}$ is decreasing on $(0,1)$. Hence, $v_\gamma$ satisfies \textit{both} \eqref{eq.Con-D} and \eqref{eq.D}. Moreover, the associated weight $w_\gamma(r):=(1-r)v_\gamma(r)=(1-r)^{\gamma+1}$ on $[0,1)$. So, it follows from Proposition \ref{P.Opt_H_v}(i) that both 
 the differentiation operator $D\colon  A^{-\gamma}\to A^{-(\gamma+1)}$ and the integration operator $J\colon  A^{-(\gamma+1)}\to A^{-\gamma}$
are continuous. These are classical results of Hardy and Littlewood; see \cite[Theorem 5.5]{Du}. These facts are well known and we will refer to them as needed.

The following  result should be compared with \cite[Proposition 5.3]{ABR_JFA}, after noting that $g\in \cB$ (resp. $g\in \cB_0$) is equivalent to  continuity of each of the operators $V_g\colon A^{-\gamma}\to A^{-\gamma}$ and $V_g\colon A^{-\gamma}_0\to A^{-\gamma}_0$ (resp. to compactness of each of the operators $V_g\colon A^{-\gamma}\to A^{-\gamma}$ and $V_h\colon A^{-\gamma}_0\to A^{-\gamma}_0$); see \cite[Theorems 1 \& 2]{BCHMP}, \cite[Proposition 3.1]{Ma}.

\begin{corollary}\label{C.Opt_Kor} Let $\gamma>0$ and  $g\in \cB$ be a non-constant function satisfying $g', \frac{1}{g'}\in H^\infty$. 
	\begin{itemize}
		\item[\rm (i)] The optimal domains  $[V_g,A^{-\gamma}]=[T_g,A^{-\gamma}]=A^{-(\gamma+1)}$ as linear spaces on $H(\D)$ and with equivalent norms.
		\item[\rm (ii)] The optimal domains  $[V_g,A^{-\gamma}_0]=[T_g,A^{-\gamma}_0]=A^{-(\gamma+1)}_0$ as linear spaces on $H(\D)$ and with equivalent norms.
	\end{itemize}
\end{corollary}

\begin{proof} For each $\gamma>0$ it was noted above that $g\in\cB$ is equivalent to $V_g\in \cL(A^{-\gamma})$ which in turn is equivalent to $V_g\in \cL(A^{-\gamma}_0)$. Accordingly, 
	part  (i) follows from  Corollary \ref{C.Opt_H_v}(i), with $v=v_\gamma$,  and 
	part	(ii)  follows from Corollary \ref{C.Opt_H_v}(ii).
\end{proof}

\begin{remark}\label{R.38}\rm 
	Observe that $A^{-\gamma}\subseteq A^{-(\gamma+1)}_0$ and that $A^{-(\gamma+1)}_0$ is a proper closed subspace of $A^{-(\gamma+1)}$, for every $\gamma>0$. Hence, the initial domain space $A^{-\gamma}$ is \textit{not dense} in the optimal domain spaces $[V_g, A^{-\gamma}]=[T_g, A^{-\gamma}]$; see Corollary \ref{C.Opt_Kor}(i). On the other hand, the initial domain space $A^{-\gamma}_0$ is \textit{dense} in the optimal domain spaces $[V_g, A^{-\gamma}_0]=[T_g, A^{-\gamma}_0]=A^{-(\gamma+1)}_0$ (cf. Corollary \ref{C.Opt_Kor}(ii)) because the space $\cP$ of all polynomials  is dense in $[V_g, A^{-\beta}_0]$ for all $\beta>0$, \cite[Proposition 3.12]{ABR_JFA}. Further examples where $X$ is not dense in $[T,X]$ occur in Section 5.
\end{remark}

\begin{corollary}\label{C.OpK_c}  Let $\gamma>0$ and $g\in \cB_0$ be a non-constant function satisfying $g', \frac{1}{g'}\in H^\infty$. Then $A^{-\gamma}$ is a proper subspace of $[V_g,A^{-\gamma}_0]=[T_g,A^{-\gamma}_0]=A^{-(\gamma+1)}_0$. 
\end{corollary}

\begin{proof}  For each $\gamma>0$ it was noted prior to Corollary \ref{C.Opt_Kor} that the condition $g\in \cB_0$ is equivalent to the compactness of $V_g\in\cL(A^{-\gamma})$ which in turn is equivalent to the compactness of $V_g\in\cL(A^{-\gamma})$. 
	Since $\lim_{r\to 1^-}v_\gamma(r)=0$, the desired conclusion
	follows from Corollary \ref{C.Opt_H_v}(iii).
\end{proof}

Corollary \ref{C.OpK_c}, which has the extra condition $g', \frac{1}{g'}\in H^\infty$, is a strengthening of Proposition 5.9(ii),(iv) in \cite{ABR_JFA},  where it is only assumed that $g\in \cB_0$.

Propositions \ref{P_H_v} and \ref{P.Opt_H_v} require some restrictions on the  weight function $v$ on $[0,1)$. We now relax these restrictions somewhat and proceed to establish a further  result regarding the description of $[V_g,H^\infty_v]$ and $[V_g,H^0_v]$.

Let $v$ be a weight function on $[0,1)$. For each $n\in\N$, let $r_n\in [0,1)$ be a  point where a maximum  of $r^nv(r)$ occurs. 
It is known that $H^\infty_v$ (resp. $H^0_v$) is isomorphic  to $\ell^\infty$ (resp. to $c_0$) if and only if the following
\begin{eqnarray*}
	&	\textrm{condition (B):}\quad  \forall b_1>1\ \exists b_2>1 \ \exists c>0\ \forall m,n\in\N:\nonumber \\
	&m\geq c,\ n\geq c,\  |m-n|\geq c \textrm{ and } \left(\frac{r_m}{r_n}\right)^m\frac{v(r_m)}{v(r_n)}\leq b_1\Rightarrow \left(\frac{r_n}{r_m}\right)^n\frac{v(r_n)}{v(r_m)}\leq b_2,
\end{eqnarray*}
is satisfied; see \cite[Theorem 1.1]{Lu_sm}. Examples of such weights $v$ on $[0,1)$ include $(1-r)^\gamma$, with $\gamma>0$, and $\exp(-(1-r)^{-1})$.

\begin{prop}\label{Fact 3} Let $g\in H(\D)$ satisfy $|g'|>0$ in $\D$ and  $v$ be a weight function on $[0,1)$. 
	\begin{itemize}
		\item[\rm (i)]  Suppose that  $\lim_{r\to1^-}v(r)=0$ and $T_g\colon H^0_v\to H^0_v$ is continuous. Then $[V_g,H^0_v]=[T_g,H^0_v]$ is a Banach space isometric to $H^0_v$. Moreover, if  $v$ also satisfies the condition (B), then  $[T_g,H^0_v]$ is a Banach space isomorphic to $c_0$.
		\item[\rm (ii)] Suppose that $T_g\colon H^\infty_v\to H^\infty_v$ is continuous. Then $[V_g,H^\infty_v]=[T_g,H^\infty_v]$ is a Banach space which is isometric to $H^\infty_v$. Moreover, if the weight function $v$ also satisfies condition (B), then  $[T_g,H^\infty_v]$ is isomorphic to $\ell^\infty$.
		\item[\rm (iii)] Suppose that  $\lim_{r\to1^-}v(r)=0$ and $T_g\colon H^0_v\to H^0_v$ is continuous. 
		\begin{itemize}
			\item[\rm (a)] $T_g\colon H^\infty_v\to H^\infty_v$ is continuous. 
			\item[\rm (b)] $(H^0_v)^{**}=H^\infty_v$.
			\item[\rm (c)] $T_g^{**}=T_g$ in $\cL(H^\infty_v)$.
			\item[\rm (d)] 	$[T_g,H^0_v]^{**}=[T_g,H^\infty_v]$.
		\end{itemize}
		\item[\rm (iv)] Let $v$ satisfy  $\lim_{r\to1^-}v(r)=0$ and the operator $T_g\colon H^0_v\to H^0_v$ be continuous but, not surjective. Then $H^\infty_v$ (resp. $H^0_v$) is properly contained in $[T_g, H^\infty_v]$ (resp. in $[T_g, H^0_v]$). 
	\end{itemize}
\end{prop}

\begin{proof} Proposition \ref{Fact 2} implies that $T_g$ is an isomorphism on $H(\D)$.
	
	(i) The first claim follows from Proposition 5.3 (iii), (iv) in \cite{ABR_JFA} and Proposition \ref{Fact 1}(i). The discussion prior to this proposition together with  condition (B) imply the second claim.

	(ii) The first claim follows from Proposition 5.3 (i), (ii) in \cite{ABR_JFA}  and Proposition \ref{Fact 1}(i). Then the discussion prior to this proposition together with  condition (B) establish the second claim.
	
	(iii) (a) Since $T_g\colon H^0_v\to H^0_v$ is continuous,  the operator $V_g\colon H^0_v\to H^0_v$ is also continuous (cf. Proposition 5.3(iii) in  \cite{ABR_JFA}). Then,  $V_g^{**}=V_g$ and $(H_v^0)^{**}=H^\infty_v$ (see Remark 3.3) imply that $V_g\colon H^\infty_v\to H^\infty_v$ is continuous and hence,  also $T_g \colon H^\infty_v\to H^\infty_v$ is continuous (cf. Proposition 5.3(i) in \cite{ABR_JFA}).
	
	(b) This was argued in the proof of part (a).
	
	(c) The operator $T_g \colon H^\infty_v\to H^\infty_v$ is continuous; see (a) above. Let $M_z\in\cL(H(\D))$ denote the operator of pointwise multiplication by $z$ (i.e., $f(z)\mapsto zf(z)$, for $f\in H(\D)$). Then $M_z\in \cL(H^0_v)$ satisfies $M_z^{**}=M_z$ (recall that $(H^0_v)^{**}=H^\infty_v$), where on the right-side $M_z\in \cL(H^\infty_v)$; see Propositions 2.1 and 2.2 in \cite{BDL}. Moreover, $M_zT_g=V_g$ in $\cL(H^0_v)$ and, as noted in the proof of (a), $V_g^{**}=V_g$ in $\cL(H^\infty_v)$. It follows,, as identities in $\cL(H^\infty_v)$ that
	\[
	V_g=V_g^{**}=(M_zT_g)^{**}=M_z^{**}T_g^{**}=M_zT_g^{**}.
	\]
	Since also $M_zT_g=V_g$ in $\cL(H^\infty_v)$, the previous identity implies that $M_zT_g=M_zT_g^{**}$ in $\cL(H^\infty_v)$.  Given $h\in H^\infty_v$, it follows that $zT_g(h)(z)=zT_g^{**}(h)(z)$, for $z\in\D$, and hence, that $T_g(h)(z)=T_g^{**}(h)(z)$, for $z\in\D\setminus \{0\}$. By analytic continuation we can conclude that $T_g(h)=T_g^{**}(h)$ in $H^\infty_v$. That is, $T_g=T_g^{**}$.
	
	(d) By  (a), (b) the space  $[T_g, H^\infty_v]$ is  defined. Part (i) implies that  $[T_g,H^0_v]$ is isometric to $H^0_v$ and hence, $[T_g,H^0_v]^{**}$ is isometric to $(H^0_v)^{**}$. Moreover,  $(H^0_v)^{**}$ is isomorphic to $H^\infty_v$ (see (b) above). But, by part (ii), $H^\infty_v$ is isometric to $[T_g,H^\infty_v]$. Accordingly, $[T_g,H^0_v]^{**}$ is isomorphic to $[T_g,H^\infty_v]$.

	(iv) Since $T_g\in \cL(H(\D))$ is an isomorphism, both $T_g\in \cL(H^0_v)$ and $T_g=T_g^{**}\in \cL(H^\infty_v)$ are injective. Also, $T_g\in \cL(H^0_v)$ is not surjective and so $0\in\sigma(T_g;H^0_v)\setminus\sigma_{pt}(T_g;H^0_v)$. The injectivity of $T_g^{**}\in \cL(H^\infty_v)$ implies that $0\not\in \sigma_{pt}(T_g^{**};H^\infty_v)$. It then follows from the well known Banach space identity
	\[
	\sigma(T_g;H^0_v)=\sigma(T_g^{**}; (H^0_v)^{**})=\sigma(T_g;H^\infty_v)
	\]
	for operators that $0\in \sigma(T_g;H^\infty_v)\setminus\sigma_{pt}(T_g;H^\infty_v)$. Hence, $T_g\in \cL(H^\infty_v)$ is not surjective. The desired conclusion now follows 
	from Proposition \ref{Fact 1}(ii).
\end{proof}

\subsection{The Hardy spaces $H^p$ for $1\leq p\leq\infty$}
Given any  $1\leq  p<\infty$, the Hardy space $H^p$ is the linear space of all  functions $f\in H(\D)$  satisfying
\begin{equation}\label{eq.pnorm}
	\|f\|_p:=\sup_{0\leq r<1}\left(\frac{1}{2\pi}\int_0^{2\pi}|f(re^{i\theta})|^p\,d\theta\right)^{1/p}<\infty
\end{equation}
and, for $p=\infty$, the linear space $H^\infty$ consists of all functions $f\in H(\D)$ such that 
\begin{equation}\label{eq.normI}
	\sup_{0\leq r<1}\left(\sup_{\theta\in [0,2\pi]}|f(re^{i\theta})|\right)<\infty.
\end{equation}
Note that the expression  \eqref{eq.normI} equals $\|f\|_\infty$; see \eqref{eq.supnorm}. 
For $1\leq p< \infty$, it is well known that $\|\cdot\|_p$ is  a norm for which $H^p$ is a Banach space, \cite[Corollary 1, p.37]{Du}.  
Whenever $1\leq  p\leq  q \leq \infty$, it turns out that $H^q\subseteq H^p$ with a continuous inclusion. 
It is known, for each $z\in\D$, that the evaluation functional $\delta_z\colon f\mapsto f(z)$, for $f\in H^p$, is continuous, that is, $\delta_z\in (H^p)^*$; see the Lemma on p.36  of \cite{Du} for $1\leq p<\infty$. For $p=\infty$ and $z\in\D$  the inequality
$|\delta_z(f)|=|f(z)|\leq \|f\|_\infty$, for $f\in H^\infty$,
shows that $\delta_z\in (H^\infty)^*$.

The space $BMOA$ consists of all functions $f\in  H^2$ such that
\[
|f(0)|+\sup_{a\in\D}\|f\circ \phi_a-f(a)\|_2<\infty,
\]
where $\phi_a$, for $a\in\D$, is the family of M\"obius automorphisms of $\D$ given by $\phi_a(z):=\frac{z-a}{1-\overline{a}z}$,  for $z\in\D$. The space $VMOA $ consists of  all functions  $f\in  BMOA$ satisfying
\[
\lim_{|a|\to 1}\|f\circ \phi_a-f(a)\|_2=0.
\]
The space $VMOA$ is the closure of the polynomials in $BMOA$, \cite[Theorem 5.5]{Gir}. In particular, $H^\infty\subseteq BMOA$ and $A(\D)\subseteq VMOA$, \cite[Theorem 5.5 and Remark 5.2]{Gir}.

For $1\leq p<\infty$,  Aleman and Siskakis, \cite{AleSi1}, proved that the operator $T_g\colon H^p\to H^p$ is continuous (resp. compact) if and only if $g\in BMOA$ (resp. $g\in VMOA$). On the other hand,  for $1\leq p\leq \infty$ the operator  $V_g\colon H^p\to H^p$ is continuous (resp. compact, resp. weakly compact) if and only if $T_g \colon H^p\to H^p$ is continuous
(resp. compact, resp. weakly compact),  \cite[Lemma 3.7]{ABR_AIOT}. In the case of continuity, we have
$\|T_g\|_{H^p\to H^p}=\|V_g\|_{H^p\to H^p}$. We point out that $V_g$ (resp. $T_g$) as defined in \cite{ABR_AIOT} corresponds to $T_g$ (resp. $V_g$) as defined here.

For $g\in BMOA$ and $1\leq p<\infty$ the optimal domain $[V_g,H^p]$ was studied in \cite{BDNS}. In particular, it was shown  there  that $([V_g,H^p], \|\cdot\|_{[V_g, H^p]})$ is a Banach space and that $H^p$ is strictly contained in $[V_g, H^p]$ with a continuous inclusion. 

Our next aim is to establish that $[V_g,H^p]=[T_g, H^p]$ with equivalent norms. First, we need to introduce some auxiliary spaces and operators  (following the approach of \cite{ABR_JFA}).
For each $1\leq p\leq\infty$,  define
\[
H_0^p:=\{f\in H^p\, :\, f(0)=0\}.
\]
Accordingly, $H_0^p=\Ker(\delta_0)$ is a closed 1-codimensional subspace of $H^p$. Moreover,  $A_0(\D):=\{f\in A(\D):\, f(0)=0\}$ is a closed 1-codimensional subspace of $A(\D)$.

Consider the \textit{forward shift} operator  $S\colon H(\D)\to H(\D)$  defined by
\[
(Sf)(z):=zf(z),\quad  z\in \D,
\]
for each $f\in H(\D)$, which belongs to $\cL(H(\D))$. The inverse operator $S^{-1}$  is the restriction to the closed subspace $H_0^p$ of $H^p$ of the \textit{backward shift} operator given by $(Bf)(z):=(f(z)-f(0))/z$, for $ z \in \D$. For further information, see \cite{CiR}.

\begin{lemma}\label{L.35} {\rm (i)} Let  $1\leq p\leq\infty$. Then $S\in \cL(H^p)$ with $\|S\|_{H^p\to H^p}=1$ and range $S(H^p)=H^p_0$. Moreover, the operator $S$ is injective with the  inverse operator $S^{-1}\colon S(H^p)\to H^p$   continuous and satisfying  
	$$
	\|S^{-1}\|_{S(H^p)\to H^p}=\|S^{-1}\|_{H^p_0\to H^p}=1.
	$$
	
{\rm (ii)}	The operator $S$ belongs to $\cL(A(\D))$, has  norm $\|S\|_{A(\D)\to A(\D)}=1$ and its range $S(A(\D))=A_0(\D)$. Moreover, $S$ is injective with the  inverse operator $S^{-1}\colon S(A(\D))\to A(\D)$   continuous and satisfying  
$$
\|S^{-1}\|_{S(A(\D))\to A(\D)}=
\|S^{-1}\|_{A(\D)_0\to A(\D)}=1.
$$
\end{lemma}

Lemma \ref{L.35} is  well known; see   \cite[Lemma 3.5]{ABR_AIOT} for a proof.

\begin{prop}\label{P.Id_p} Let $1\leq p\leq\infty$ and $g\in H(\D)$ have the property that $V_g\colon H^p\to H^p$ is continuous. Then $[V_g,H^p]=[T_g,H^p]$ as linear spaces in $H(\D)$ and with isometric norms. In particular, $([T_g,H^p],\|\cdot\|_{[T_g,H^p]})$ is a Banach space.
	\end{prop}

\begin{proof}
	Since $V_g\colon H^p\to H^p$ is continuous, it was noted above that $T_g\colon H^p\to H^p$ is  also  continuous.
	
	Let $h\in [V_g,H^p]$. Then $h\in H(\D)$ and $V_gh\in H^p$. According to the definition we have $(V_gh)(0)=0$ and so actually  $V_gh\in H^p_0$. In view of \eqref{eq.CC} and the definition of $V_g$,
	 Lemma \ref{L.35} implies that  $T_gh=(S^{-1}\circ V_g)h\in H^p$. Moreover,
	\begin{align}\label{eq.1}
	\|h\|_{[T_g,H^p]}&=\|T_gh\|_{H^p}=\|(S^{-1}\circ V_g)h\|_{H^p}\leq \|S^{-1}\|_{H^p_0\to H^p}\|V_gh\|_{H^p}\nonumber\\
	&=\|V_gh\|_{H^p}=\|h\|_{[V_g,H^p]}<\infty.
	\end{align}
This implies that $h\in [T_g,H^p]$. Therefore $[V_g,H^p]\subseteq [T_g,H^p]$ with a continuous inclusion.

Let $f\in [T_g,H^p]$. Then $f\in H(\D)$ and $T_gf\in H^p$. By Lemma \ref{L.35}(i) it follows that $V_gf=(S\circ T_g)f\in H^p_0\subseteq H^p$. So, $f\in [V_g,H^p]$. Moreover,
\begin{align}\label{eq.2}
	\|f\|_{[V_g,H^p]}&=\|V_gf\|_{H^p}=\|(S\circ T_g)f\|_{H^p}\leq \|S\|_{H^p\to H^p_0}\|T_gh\|_{H^p}\nonumber\\
	&=\|S\|_{H^p\to H^p}\|T_gh\|_{H^p}=\|f\|_{[T_g,H^p]}<\infty.
\end{align}
Accordingly, $f\in [V_g,H^p]$. Therefore, $[T_g,H^p]\subseteq [V_g,H^p]$ with a continuous inclusion.

The inequalities \eqref{eq.1} and \eqref{eq.2} imply that  $[T_g,H^p]= [V_g,H^p]$ as linear spaces in $H(\D)$ and  with isometric norms.
Since $([V_g,H^p], \|\cdot\|_{[V_g, H^p]})$ is a Banach space (cf. \cite{BDNS} for $1\leq p<\infty$, Remark \ref{R.Hinfinito} for $p=\infty$), it follows that $([T_g,H^p], \|\cdot\|_{[V_g, H^p]})$ is  also   a Banach space. \end{proof}

\begin{prop}\label{P.Hp_inclusion} Let $1\leq p\leq\infty$ and $g\in H(\D)$ have the property that $V_g\colon H^p\to H^p$ is compact. Then $H^p$ is a proper subspace of $[V_g,H^p]=[T_g,H^p]$.
	\end{prop}

\begin{proof} Since $([V_g,H^p], \|\cdot\|_{[V_g,H^p]}))$ is a Banach space and $V_g\colon H^p\to H^p$ is compact, we can apply Corollary \ref{C1}(iii) to conclude that  $H^p$ is a proper subspace of $[V_g,H^p]$.
\end{proof}

The following result provides more specific information for the spaces $[V_g,H^p]$ and $[T_g,H^p]$.

\begin{prop}\label{Fact 4} Let $1\leq p\leq \infty$. Suppose that $g\in H(\D)$ satisfies $|g'|>0$ in $\D$ and that $T_g\colon H^p\to H^p$ is continuous. 
\begin{itemize}
	\item[\rm (i)] The optimal domain space $[V_g,H^p]=[T_g,H^p]$ is a Banach space isometric to $H^p$. So, $[T_g,H^p]$ is separable and uniformly convex for $p\in (1,\infty)$. In particular, $[V_g,H^p]=[T_g,H^p]$ is reflexive for $p\in (1,\infty)$.
	\item[\rm (ii)] Suppose that  $T_g\colon H^p\to H^p$ is not surjective. Then $H^p$ is strictly contained in $[T_g,H^p]$.
\end{itemize}
\end{prop}

\begin{proof} (i) Proposition \ref{Fact 2} implies that $T_g\in \cL(H(\D))$ is an isomorphism. Moreover, Lemma 3.7 in \cite{ABR_AIOT} shows that $V_g\colon H^p\to H^p$ is also continuous.
	It follows from Propositions \ref{Fact 1}(i) and \ref{P.Id_p}  that $[V_g,H^p]=[T_g,H^p]$ is a Banach space isometric to $H^p$.
	
	For $\mathbb{T}:=\{z\in\C:\ |z|=1\}$ the Banach space $L^p(\mathbb{T})$ is uniformly convex whenever $1<p<\infty$ and so $H^p$, being isometrically isomorphic to a closed subspace of $L^p(\mathbb{T})$, \cite[Theorem III.3.8]{Ka}, is also uniformly convex. Since uniformly convex Banach spaces are necessarily reflexive, \cite{Pe}, the well known fact that the spaces $H^p$, for $1<p<\infty$, are  reflexive follows.
	
	Part (ii) follows from  Proposition \ref{Fact 1}(ii).
\end{proof}

The following result for $A(\D)$ is the analogue of \cite[Lemma 3.7]{ABR_AIOT} for $H^p$.

\begin{prop}\label{P.A} Let $g\in H(\D)$. 
	\begin{itemize}
		\item[\rm (i)] The operator $V_g\colon A(\D)\to A(\D)$ is continuous (resp. compact, resp. weakly compact) if and only if the operator $T_g\colon A(\D)\to A(\D)$ is continuous (resp. compact, resp. weakly compact). 
		\item[\rm (ii)] Let $V_g\colon A(\D)\to A(\D)$ be continuous. Then $[V_g,A(\D)]=[T_g,A(\D)]$ as linear spaces in $H(\D)$ and with isometric norms. In particular, $([V_g,A(\D)], \|\cdot\|_{[V_g,A(\D)]})$  is a Banach space and hence, $([T_g,A(\D)], \|\cdot\|_{[T_g,A(\D)]})$ is also a Banach space.
	\end{itemize}
\end{prop}

\begin{proof} (i) Suppose that $T_g\colon A(\D)\to A(\D)$ is continuous. By Lemma \ref{L.35}(ii) also $V_g=S\circ T_g\colon A(\D)\to A_0(\D)\subseteq A(\D)$ is continuous.
	
	Conversely, suppose that $V_g\colon A(\D)\to A(\D)$ is continuous. Since $V_g(A(\D))\subseteq A_0(\D)$,
	 Lemma \ref{L.35}(ii) implies that $T_g = S^{-1} \circ V_g\colon A(\D) \to A(\D)$ is also continuous.
	
The proof of compactness (weak compactness)	follows in similar way.
	
	 (ii) The proof  that $[V_g,A(\D)]=[T_g,A(\D)]$ as linear spaces in $H(\D)$ and with isometric norms follows by arguing as in the proof of Proposition \ref{P.Id_p}.

It remains to show that  	 $([V_g,A(\D)], \|\cdot\|_{[V_g,A(\D)]})$ is a Banach space. Since
		$V_g\colon A(\D)\to A(\D)$ is  continuous,  by \cite[Theorem 2.5(i), (ii)]{Cont-a}  the operator $V_g\colon H^\infty\to H^\infty$ is also continuous. So, Remark \ref{R.Hinfinito}(ii) ensures that $([V_g, H^\infty],\|\cdot\|_{[V_g,H^\infty]})$ is a Banach space. Moreover, $A(\D)$ is a closed subspace of $H^\infty$ satisfying $V_g(A(\D))\subseteq A(\D)$. We can then  apply  \cite[Proposition 2.4]{ABR_JFA} to conclude  that $([V_g,A(\D)], \|\cdot\|_{[V_g,A(\D)]})$  is a closed subspace of $([V_g, H^\infty],\|\cdot\|_{[V_g,H^\infty]})$ and hence, that it is a  Banach space.
\end{proof}

\begin{prop}\label{P.Hinfinito_inclusione} Let $g\in H(\D)$ and $V_g\colon A(\D)\to A(\D)$ be compact. Then $H^\infty$ is a proper subspace of $[V_g,A(\D)]=[T_g,A(\D)]$.
\end{prop}

\begin{proof} Since $V_g\colon A(\D)\to A(\D)$ is compact, Theorem 1.7 of \cite{Cont-a} implies that  $V_g(H^\infty)\subseteq A(\D)$ and that the operator $V_g\colon H^\infty\to A(\D)$ is also compact (as is $V_g\colon H^\infty\to H^\infty$). According to Proposition \ref{P.A}(ii) we have that $([V_g,A(\D)],\|\cdot\|_{[V_g,A(\D)]})$ is a Banach space. So, by Corollary \ref{C1}(iii) we can conclude that $H^\infty$ is a proper subspace of $[V_g,A(\D)]$.
\end{proof}

\begin{prop}\label{P.A_H} Let $g\in H(\D)$ satisfy $g', \frac{1}{g'}\in H^\infty$ and  $V_g\colon A(\D)\to A(\D)$ be continuous.
	\begin{itemize}
		\item[\rm (i)] $H^1\subseteq [V_g, A(\D)]=[T_g, A(\D)]\subseteq A^{-1}_0$.
		\item[\rm (ii)] $H^1\subseteq [V_g, H^\infty]=[T_g, H^\infty]\subseteq A^{-1}$.
	\end{itemize} 
\end{prop}

\begin{proof} (i) Let $f\in H^1\subseteq H(\D)$. Since $g'\in H^\infty$, it follows that also $fg'\in H^1$ (cf. \eqref{eq.pnorm}). According to \eqref{eq.BB},  the function $J(fg')\in H(\D)$ satisfies $D(J(fg'))=fg'\in H^1$. Then by \eqref{eq.C} and \cite[Theorem 3.11]{Du} we can conclude that $V_g(f)=J(fg')\in A(\D)$. This means that $f\in [V_g,A(\D)]$. So, $H^1\subseteq [V_g, A(\D)]$.
	
	Next, let $h\in [V_g, A(\D)]$. Then $h\in H(\D)$ and, via \eqref{eq.C}, we have $J(hg')=V_g(h)\in A(\D)$. The Stone-Weierstrass theorem implies that $\cP|_{\overline{\D}}$ is dense in $C(\overline{\D})$ and hence, $\cP$ is dense in $(A(\D),\|\cdot\|_\infty)$. Since $H^\infty\subseteq \cB$ continuously (see (2.3) in Section 2.1 of \cite{ACP}), also $A(\D)\subseteq \cB$ continuously. But, $\cB_0$ is the closure of $\cP$ in $\cB$, \cite[Theorem 2.1]{ACP}, and so 
	 $A(\D)\subseteq\cB_0$. Accordingly,  $J(hg')\in \cB_0$ and hence,  $hg'=D(J(hg'))\in A^{-1}_0$. Since $\frac{1}{g'}\in H^\infty$, it follows that also  $h=\frac{1}{g'}\cdot(hg')\in A^{-1}_0$.
	 
	 The equality $[V_g,A(\D)]=[T_g,A(\D)]$ is due to Proposition \ref{P.A}(ii).
	
	(ii) First note that the operator $V_g\colon H^\infty\to H^\infty$ is also continuous, \cite[Theorem 2.5(i), (ii)]{Cont-a}. 
	
	The inclusion $H^1\subseteq [V_g, H^\infty]$ follows from the fact that $H^1\subseteq [V_g, A(\D)]$ (cf. part (i)) and that $[V_g, A(\D)]\subseteq [V_g, H^\infty]$ because $A(\D)\subseteq H^\infty$.
	
	Next, let $f\in  [V_g, H^\infty]$. Then $f\in H(\D)$ and $J(fg')=V_g(f)\in H^\infty\subseteq \cB$ and hence, $fg'=D(J(fg'))\in A^{-1}$.  Since $\frac{1}{g'}\in H^\infty$, it follows that also $f=\frac{1}{g'}\cdot(fg')\in A^{-1}$.
	
	 The equality $[V_g,H^\infty]=[T_g,H^\infty]$ is due to Proposition \ref{P.Id_p}.
	\end{proof}

For $0<\alpha\leq 1$  recall the \textit{Lipschitz spaces} of analytic functions on $\D$ defined by
\[
\Lambda_\alpha=:\left\{g\in H(\D):\ \sup_{z\in \D}(1-|z|)^{1-\alpha}|g'(z)|<\infty\right\},
\] 
\[
\lambda_\alpha=:\left\{g\in H(\D):\ \lim_{|z|\to 1^- }(1-|z|)^{1-\alpha}|g'(z)|=0\right\};
\] 
see \cite[Chapter 5]{Du}, for example. Observe that $g\in\Lambda_\alpha$ if and only if $g'\in A^{1-\alpha}$.

\begin{prop}\label{P.HpHq} Let $g\in H(\D)$ and $1\leq p<q$.
	\begin{itemize}
		\item[\rm (i)] Suppose that  $\left(\frac{1}{p}-\frac{1}{q}\right)\leq 1$ and $g\in \Lambda_{\frac{1}{p}-\frac{1}{q}}$. Then $H^p\subseteq [V_g,H^q]=[T_g,H^q]$ with a continuous inclusion.
		\item[\rm (ii)] Suppose that  $\left(\frac{1}{p}-\frac{1}{q}\right)<1$ and $g\in \lambda_{\frac{1}{p}-\frac{1}{q}}$. Then $H^p$ is a proper subspace of $[V_g,H^q]=[T_g,H^q]$.
	\end{itemize}
\end{prop}

\begin{proof} (i) The conditions $\left(\frac{1}{p}-\frac{1}{q}\right)\leq 1$ and $g\in \Lambda_{\frac{1}{p}-\frac{1}{q}}$ imply that $V_g\colon H^p\to H^q$ is continuous, \cite[Theorem 1(iii)]{AC}. Since $p<q$, we have $H^q\subseteq H^p$ with a continuous inclusion and so $V_g\colon H^q\to H^q$ is also continuous. Accordingly,  $([V_g,H^q],\|\cdot\|_{[V_g,H^q]})$ is defined and is a Banach space; see Proposition \ref{P.Id_p}.  The definition of the optimal domain space $[V_g,H^q]$ shows that $H^p\subseteq [V_g,H^q]$.
	
	(ii) The conditions $\left(\frac{1}{p}-\frac{1}{q}\right)< 1$ and $g\in \lambda_{\frac{1}{p}-\frac{1}{q}}$ imply that $V_g\colon H^p\to H^q$ is compact, \cite[Theorem 1(iii)]{AC}. So, as observed in part (i), we can define the  Banach space $([V_g,H^q],\|\cdot\|_{[V_g,H^q]})$. By Corollary \ref{C1}(iii) we can conclude that $H^p$ is a proper subspace of $[V_g,H^q]$.
\end{proof}

\section{Generalized Cesàro operators $C_t$ for $0\leq t\leq 1$}

This section is devoted to a consideration of the optimal domain of a particular class of classical operators. Namely, the operators $V_{g_t}$ and $T_{g_t}$, for $t\in [0,1]$, where the family of symbols $\{g_t:\ t\in [0,1]\}$ is specified in \eqref{eq.17}. Since
%
 \begin{equation}\label{eq.NN}
 	|g'_t(z)|=\left|\frac{1}{1-tz}\right|\leq \frac{1}{|1-|tz||}\leq \frac{1}{1-t},\quad z\in\D,
 \end{equation}
 it follows that
$g_t\in \cB_0$. 
  For each $f\in H(\D)$, the operator $C_t$ is defined by $(C_tf)(0):=f(0)$ and
\begin{equation}\label{eq.CesaroO}
(C_tf)(z):=\frac{1}{z}\int_0^z\frac{f(\xi)}{1-t\xi}\,d\xi.\ z\in\D\setminus\{0\}.
\end{equation}
Note that $C_t=T_{g_t}$. It is called the  \textit{generalized Cesàro operator} corresponding to $t$ (for $t=1$ it is the classical  \textit{Cesàro operator}). We  point out that the related operator $V_{g_t}$ is given by
\begin{equation}\label{eq.Volt_C}
(V_{g_t}f)(z)=\int_0^z\frac{f(\xi)}{1-t\xi}\,d\xi,\quad z\in\D.
\end{equation}
Aleman and Persson have made an extensive investigation of
various properties of generalized Ces\`aro operators
acting in a large class of Banach spaces of analytic functions on $\D$, \cite{AP,AleSi1,P}.
In particular, their results apply to the classical Ces\`aro operator $C$ given by \eqref{eq.CesaroO} when it acts in the Korenblum growth spaces $A^{-\gamma}$ and $A^{-\gamma}_0$ for $\gamma > 0$.
Additional results which complement and extend their work can be found in \cite{ABR-R}.
For a detailed investigation of the optimal domain spaces $[C,H^p]$, for $1\leq p<\infty$, we refer to \cite{CR}. 

The following result  collects some relevant properties of the functions $g_t$, for $t\in [0,1]$.

\begin{lemma}\label{P.Propgt} Let $0\leq t<1$.
	\begin{itemize}
		\item[\rm (i)] The function $g_t\in A(\D)\subset VMOA$.
		\item[\rm (ii)] The function $g'_t\in A(\D)\subset VMOA$.
		\item[\rm (iii)] The function $\frac{1}{g'_t}\colon z\mapsto 1-tz\in A(\D)\subseteq VMOA$.
		\item[\rm (iv)] For each $0<\alpha\leq 1$, the function $g_t\in \Lambda_\alpha$, that is, $g'_t\in A^{-(1-\alpha)}$.
			\item[\rm (v)] For each $0<\alpha< 1$, the function $g_t\in \lambda_\alpha$, that is, $g'_t\in A^{-(1-\alpha)}_0$.
			\item[\rm (vi)] $g_t\in \cB_0$.
	\end{itemize}
	\end{lemma}

\begin{proof} (i) Since $g_t$ is holomorphic in the open disc centred at $0$ with radius $\frac{1}{t}>1$, which contains $\overline{\D}$, it is clear that $g_t\in A(\D)$. A similar argument as in (i) applies in (ii) to $g'_t(z)=\frac{1}{1-tz}$, for $z\in\D$, as is the case for the polynomial $\frac{1}{g'_t}$ in (iii). Since $(1-\alpha)\geq 0$ and $g'_t$ satisfies \eqref{eq.NN}, part (iv) follows from the definition of $\Lambda_\alpha$. For part (v) a similar argument as in (iv) is valid from the definition of $\lambda_\alpha$ as now $(1-\alpha)>0$. Finally, part (vi) is clear from \eqref{eq.NN} and \eqref{eq.lim}.
	\end{proof}

\begin{prop}\label{P.C_tHv} Let $v\colon [0,1)\to (0,\infty)$ be a log-convex weight satisfying $\lim_{r\to 1^-}v(r) =0$ and  for which the differentiation operator $D\colon H^\infty_v\to H^\infty_w$ and the integration operator $J\colon H^\infty_w\to H^\infty_v$ are continuous, where  $w(r):=(1-r)v(r)$, for $r\in [0,1)$. For each $t\in [0,1)$, the following statements are valid.
	\begin{itemize}
		\item[\rm (i)] $[V_{g_t},H^\infty_v]=[C_t, H^\infty_v]=H^\infty_w$ as Banach spaces.
		\item[\rm (ii)] $[V_{g_t},H^0_v]=[C_t, H^0_v]=H^0_w$ as Banach spaces.
		\item[\rm (iii)] $H^\infty_v$ is a proper subspace of $[C_t,H^0_v]$ and hence, so is $H^0_v\subseteq H^\infty_v$.
	\end{itemize}
\end{prop}

\begin{proof} It is known that $C_t\colon H^0_v\to H^0_v$ is compact, \cite[Proposition 2.7]{ABR_CM}. Hence, the result follows from Corollary \ref{C.Opt_H_v} and Lemma \ref{P.Propgt}(ii) and (iii). 
\end{proof}

\begin{remark}\label{R.43}\rm 
	(a) Recall, if $v$ satisfies both \eqref{eq.Con-D} and \eqref{eq.D}, then both $D$ and $J$ (in Proposition \ref{P.C_tHv}) are necessarily continuous.
	
	(b) Concerning part (ii) of Proposition \ref{P.C_tHv} (with $v$ as given there), set $X:=H^0_v$, $Y:=\cP$ the space of all polynomials on $\D$ and $T:=C_t$, for $t\in (0,1)$. Then $Y\subseteq X$ and, as noted earlier, $Y$ is dense in $X$. Since $Y$ is also dense in $H^0_w=[T,X]$, it follows from Proposition \ref{P.Denso} that $H^0_v$ is \textit{dense} in $[C_t,H^0_v]$.
\end{remark}

\begin{prop}\label{P.Kor} Let $t\in [0,1)$ and $\gamma>0$.
	\begin{itemize}
		\item[\rm (i)] $[V_{g_t},A^{-\gamma}]=[C_t, A^{-\gamma}]=A^{-(\gamma+1)}$ and $[V_{g_t},A^{-\gamma}_0]=[C_t, A^{-\gamma}_0]=A^{-(\gamma+1)}_0$ as Banach spaces.
		\item[\rm (ii)] $A^{-\gamma}$ is a proper  subspace of $[C_t, A^{-\gamma}_0]=A_0^{-(\gamma+1)}$.
		\end{itemize}
\end{prop}

\begin{proof} The stated results follow from Corollaries \ref{C.Opt_Kor} and \ref{C.OpK_c} and Lemma \ref{P.Propgt}.
\end{proof}

Concerning the classical Cesàro operator we have the following result.

\begin{prop}\label{P.Kor1} Let  $\gamma>0$.
	\begin{itemize}
		\item[\rm (i)] $[V_{g_1},A^{-\gamma}]=[C_1, A^{-\gamma}]$ and $[V_{g_1},A^{-\gamma}_0]=[C_1, A^{-\gamma}_0]$ as Banach spaces.
		\item[\rm (ii)] $A^{-\gamma}$ is a proper  subspace of $[C_1, A^{-\gamma}]$ and $A^{-\gamma}_0$ is a proper  subspace of $[C_1, A^{-\gamma}_0]$.
	\end{itemize}
\end{prop}

\begin{proof} Recall that $A^{-\gamma}$ (resp. $A^{-\gamma}_0$) corresponds to the weight function $v_\gamma(r)=(1-r)^\gamma$ on $[0,1)$ for $H^\infty_v$ (resp. $H^0_v$). Since $|g'_1|>0$ on $\D$, the result follows from Proposition \ref{Fact 3}, after noting  that both the operators $C_1\colon A^{-\gamma}\to A^{-\gamma}$ and   $C_1\colon A^{-\gamma}_0\to A^{-\gamma}_0$ are continuous and $0\in \sigma(C_1; A^{-\gamma})=\sigma(C_1,A^{-\gamma}_0)$; see \cite{AP}, \cite{P}. In particular, $C_1\colon A^{-\gamma}_0\to A^{-\gamma}_0$ is not surjective (as it is injective).
\end{proof}

Concerning $C_t$ acting in $H^p$ we have the following result for its optimal domain.

\begin{prop}\label{P.Hpp} Let $1\leq p<\infty$.
	\begin{itemize}
		\item[\rm (i)] $[V_{g_t},H^p]=[C_t,H^p]$ as Banach spaces for all $t\in [0,1]$.
		\item[\rm (ii)] $H^p$ is a proper subspace of $[C_t,H^p]$ for all $t\in [0,1]$.
		\item[\rm (iii)] $H^1\subseteq [C_t,H^p]\subseteq A^{-(\frac{1}{p}+1)}_0$ for all $1<p<\infty$ and $t\in [0,1)$. Moreover, both inclusions are strict.
		\end{itemize}
\end{prop}
\begin{proof}
	(i)  For $t=1$ it is known that $C_1\colon H^p\to H^p$ is continuous, \cite[Theorem 1]{Sis}. Given $t\in [0,1)$ we can apply
	Proposition \ref{P.Id_p}, after recalling that $C_t\colon H^p\to H^p$ is continuous because $g_t\in VMOA$; see Lemma \ref{P.Propgt} and the discussion prior to Lemma \ref{L.35}.
	
(ii) For $t\in [0,1)$,	 we apply Proposition \ref{P.Hp_inclusion}, after noting that $C_t\colon H^p\to H^p$ is compact because $g_t\in VMOA$. For $t=1$, since $0\in \sigma(C_1, H^p)$, \cite[Theorem 1]{Sis}, the result follows from Proposition \ref{Fact 4}. 
	
	(iii) The fact that $H^1$ is a proper subspace of $[C_t,H^p]=[V_{g_t},H^p]$ follows from Proposition \ref{P.HpHq}(ii) (with $q$ replaced by $p$ there and setting $p=1$ there), after  observing that $g_t\in \lambda_{1-\frac{1}{p}}$ as $(1-\frac{1}{p})<1$ (cf. Lemma \ref{P.Propgt}(v)).
			
	Next, if $f\in [C_t, H^p]=[V_{g_t},H^p]$, then $f\in H(\D)$ and $V_{g_t}f\in H^p$. So, by \cite[Theorem 5.9]{Du} the function $V_{g_t}f\in A^{-\frac{1}{p}}_0$. Accordingly, $f\in [V_{g_t},A^{-\frac{1}{p}}_0]=A^{-(\frac{1}{p}+1)}_0$ (cf. Proposition \ref{P.Kor}(i)). This shows that $[C_t,H^p]\subseteq A^{-(\frac{1}{p}+1)}_0$. This inclusion is strict because  $[C_t,H^p]$ is reflexive as $1<p<\infty$ (cf. Proposition \ref{Fact 4}(i)) whereas $A^{-(\frac{1}{p}+1)}_0$ is isomorphic to $c_0$ (cf. Proposition \ref{Fact 3}(i)).
\end{proof}


\begin{remark}\rm By considering $C_1\mathbbm{1}$ it is clear that $C_1$ \textit{fails} to map $H^\infty$ into itself. However, for $p=\infty$, parts (i) and (ii) of Proposition \ref{P.Hpp} remain valid for each $t\in [0,1)$. The proof is as follows.
	Since $C_t\colon H^\infty\to H^\infty$ is compact, \cite[Proposition 2.2]{ABR_AIOT}, we can conclude via \cite[Lemma 3.7]{ABR_AIOT} that $V_{g_t}\colon H^\infty\to H^\infty$ is also compact. Hence, Proposition \ref{P.Id_p} implies that $[V_{g_t},H^\infty]=[C_t,H^\infty]$ as Banach spaces (which is part (i) of Proposition \ref{P.Hpp} for $p=\infty$) and Proposition \ref{P.Hp_inclusion} yields that $H^\infty$ is a proper subspace of $[C_t,H^\infty]$. That is, part (ii) of Proposition \ref{P.Hpp} is also valid for $p=\infty$.
\end{remark}

\begin{prop}\label{P.A_Ct}Let $t\in [0,1)$.
	\begin{itemize}
		\item[\rm (i)] $[V_{g_t},A(\D)]=[C_t,A(\D)]$ as Banach spaces.
		\item[\rm (ii)] $H^\infty$ is a proper subspace of $[C_t,A(\D)]$.
		\item[\rm (iii)] $H^1\subseteq [C_t,A(\D)]\subseteq A^{-1}_0$.
		\item[\rm (iv)] $H^1\subseteq [C_t, H^\infty]\subseteq A^{-1}$.
		\end{itemize}
	\end{prop}

\begin{proof}
	The result follows from Propositions \ref{P.A}, \ref{P.Hinfinito_inclusione} and \ref{P.A_H}, after  observing that $C_t\colon H^\infty\to H^\infty$ and $C_t\colon A(\D)\to A(\D)$ are compact, \cite[Proposition 2.2]{ABR_AIOT}, that $C_t(H^\infty)\subseteq A(\D)$ by \cite[Proposition 3.10]{ABR_AIOT},  and that $g'_t, \frac{1}{g'_t}\in H^\infty$.
\end{proof}

\section{Finite dimensional kernel}

We begin with a result which is pure linear algebra.

\begin{prop}\label{P.LA} Let $E\not=\{0\}$ be any  linear  space over $\C$ and $T\colon E\to E$ be any linear map such  that $T$ is not injective.
	\begin{itemize}
		\item[\rm (i)] $T^{-1}(\Ker(T))=T^{-1}(\Ker(T)\cap {\rm Im}(T))=\Ker(T^2)$.
		\item[\rm (ii)] $\Ker(T)\subseteq T^{-1}(\Ker(T))$.
		\item[\rm (iii)] Suppose that $\Ker(T)\not=\{0\}$ is  finite dimensional. Then also $T^{-1}(\Ker(T))$ is  finite dimensional. More precisely, 
		\begin{itemize}
			\item[\rm (a)] if $\Ker(T)\cap {\rm Im}(T)=\{0\}$, then $T^{-1}(\Ker(T))=\Ker(T)$  and
			\item[\rm (b)]  if $\Ker(T)\cap {\rm Im}(T)\not=\{0\}$, then 
			$$
			{\rm dim}\,\Ker(T)<{\rm dim}\,T^{-1}(\Ker(T))={\rm dim}\,\Ker(T)+{\rm dim}\,(\Ker(T)\cap {\rm Im}(T)).
			$$
		\end{itemize}
	\end{itemize}
\end{prop} 

\begin{proof}(i) Of course, $T^{-1}(\Ker(T)\cap {\rm Im}(T))\subseteq T^{-1}(\Ker(T))$. Conversely, let $x\in T^{-1}(\Ker(T))$. Then $Tx\in \Ker(T)$. Moreover, also $Tx\in {\rm Im}(T)$. Accordingly, $Tx\in \Ker(T)\cap {\rm Im}(T)$ and hence, $x\in T^{-1}(\Ker(T)\cap {\rm Im}(T))$. We can conclude that $T^{-1}(\Ker(T))=T^{-1}(\Ker(T)\cap {\rm Im}(T))$. 
	
	Now, let $x\in T^{-1}(\Ker(T))$. Then $Tx\in \Ker(T)$ and so also $T^2x=0$,  that is, $x\in \Ker(T^2)$. On the other hand, if $x\in \Ker(T^2)$, then $T^2x=T(Tx)=0$. This means that $Tx\in \Ker(T)$ from which it follows that $x\in T^{-1}(\Ker(T))$. Hence, $T^{-1}(\Ker(T))=\Ker (T^2)$.
	
	(ii) By part (i) we have $T^{-1}(\Ker(T))=\Ker(T^2)$ and so the conclusion follows from $\Ker(T)\subseteq \Ker(T^2)$.
	
	(iii) (a) The condition $\Ker(T)\cap {\rm Im}(T)=\{0\}$ implies, by part (i), that
	\[
	T^{-1}(\Ker(T))=T^{-1}(\{0\})=\Ker(T).
	\] 
	In particular, $T^{-1}(\Ker(T))$ is finite dimensional.
	
	(b) Set $m:={\rm dim}\,(\Ker(T)\cap {\rm Im}(T))$ and $n:={\rm dim}\,\Ker(T)$, in which case $1\leq m\leq n$. Let $\{y_i\}_{i=1}^m$ be a \textit{basis} for $\Ker(T)\cap {\rm Im}(T)$. The inclusion $\Ker(T)\cap {\rm Im}(T)\subseteq {\rm Im}(T)$ implies that there exist elements $\{x_i\}_{i=1}^m\subseteq E\setminus\{0\}$ satisfying $Tx_i=y_i$, for $1\leq i\leq m$. Since the elements $\{y_i\}_{i=1}^m$ are all distinct, also $\{x_i\}_{i=1}^m$ consists of distinct elements. Moreover, $Tx_i=y_i\not=0$, for all $1\leq i\leq m$, implies that $\{x_i\}_{i=1}^m\subseteq E\setminus\Ker(T)$.
	
	Suppose that the linear combination
	\begin{equation}\label{eq.LC}
		\alpha_1x_1+\ldots +\alpha_mx_m+\beta_1y_1+\ldots +\beta_my_m=0.
	\end{equation}
	Apply the linear map $T$ yields 
	\[
	\alpha_1 Tx_1+\ldots +\alpha_mTx_m=0,
	\]
	where we have used $Ty_i=0$ for $1\leq i\leq m$ as  $\{y_i\}_{i=1}^m\subseteq \Ker(T)$. Since $Tx_i=y_i$, for all $1\leq i\leq m$, we have $\alpha_1y_1+\ldots +\alpha_my_m=0$. Then the  linear independence of $\{y_i\}_{i=1}^m$ yields that $\alpha_1=\alpha_2=\ldots=\alpha_m=0$ and so it follows from \eqref{eq.LC} that $\beta_1y_1+\ldots +\beta_my_m=0$. Hence, also $\beta_1=\beta_2=\ldots=\beta_m=0$. Accordingly, the vectors $\{x_i\}_{i=1}^m\cup\{y_i\}_{i=1}^m$ are all distinct (if $x_i=y_j$, then $Tx_i=Ty_j=0$ and so $x_i\in \Ker(T)$; a contradiction) and linear independent. The claim is that
	\begin{equation}\label{eq.Int}
		\Ker(T)\cap {\rm span}\{x_i\}_{i=1}^m=\{0\}.
	\end{equation}
	Indeed, let $x\in 	\Ker(T)\cap {\rm span}\{x_i\}_{i=1}^m$. Then $Tx=0$ and there exist complex numbers $\{\gamma_i\}_{i=1}^m$ satisfying $x=\sum_{i=1}^m\gamma_ix_i$. Apply $T$ yields that $0=Tx=\sum_{i=1}^m\gamma_iTx_i=\sum_{i=1}^m\gamma_iy_i$. By the linear independence of $\{y_i\}_{i=1}^m$ we can conclude that $\gamma_i=0$, for all $i=1,\ldots,m$, that is, $x=0$.  Hence, \eqref{eq.Int} has been established.
	
	Define the subspace $V:=\Ker(T)+ {\rm span}\{x_i\}_{i=1}^m$. It folllows from \eqref{eq.Int} that
	\begin{equation}\label{eq.Dim}
		{\rm dim}\, V={\rm dim}\,\Ker(T)+ {\rm dim}(\Ker(T)\cap {\rm Im}(T)).
	\end{equation}
	The claim is that 
	\begin{equation}\label{eq.V}
		V=T^{-1}(\Ker(T)).
	\end{equation}
	Indeed, we have by part (ii) that $\Ker(T)\subseteq T^{-1}(V)$. Also, $Tx_i=y_i$, for all $i=1,\ldots,m$, with $\{y_i\}_{i=1}^m\subseteq  \Ker(T)$, implies that $\{x_i\}_{i=1}^m\subseteq  T^{-1}(\Ker(T))$ and so ${\rm span}\{x_i\}_{i=1}^m\subseteq  T^{-1}(\Ker(T))$. It is then clear that $V\subseteq T^{-1}(\Ker(T))$.
	
	Conversely, let $x\in T^{-1}(\Ker(T))$. By part (i) we have that $x\in T^{-1}(\Ker(T)\cap {\rm Im}(T))$, i.e., $Tx\in \Ker(T)\cap {\rm Im}(T)$. So, there exist $\alpha_i\in\N$, for $1\leq i\leq m$, such that $Tx=\sum_{i=1}^m\alpha_iy_i$. Hence,
	\[
	T\left(x-\sum_{i=1}^m\alpha_ix_i\right)=Tx-\sum_{i=1}^m\alpha_iTx_i=Tx-\sum_{i=1}^m\alpha_iy_i=Tx-Tx=0.
	\]
	Accordingly, $x-\sum_{i=1}^m\alpha_ix_i\in \Ker(T)$ and hence,
	\[
	x=\left(x-\sum_{i=1}^m\alpha_ix_i\right)+\left(\sum_{i=1}^m\alpha_ix_i\right)\in \Ker(T)+{\rm span}\{x_i\}_{i=1}^m=V.
	\]
	This completes the proof of \eqref{eq.V} and hence, via \eqref{eq.Dim}, also the proof of (b).
\end{proof}

\begin{remark}\label{Rem}\rm  Let $E\not=\{0\}$ be a linear space over $\C$ and $T\colon E\to E$ be a linear map. Suppose that $\Ker(T)$ is finite dimensional and satisfies
	\begin{equation}\label{eq.INcc}
		\{0\}\not=\Ker(T)\subseteq {\rm Im}(T).
	\end{equation}
	Of course, a sufficient condition for \eqref{eq.INcc} to hold is that $T\colon E\to E$ is surjective.
	
	It follows from \eqref{eq.INcc} that $\Ker(T)\cap {\rm Im}(T)=\Ker(T)$ and hence, part (iii) (b) of  Proposition \ref{P.LA} implies that
	\begin{equation}\label{eq.Dim2}
		{\rm dim}\, T^{-1}(\Ker(T))=2 {\rm dim}\, \Ker(T).
	\end{equation}
	Setting $X:=\Ker(T)$, observe that $T(X)\subseteq X$ and that the space
	\[
	T^{-1}(X)=T^{-1}(\Ker(T))=\{x\in E:\ Tx\in X\}
	\]
	is what we   call the \textit{optimal domain} $[T,X]$ of $T\colon X\to X$. This would be   genuinely the case if $E=H(\D)$ and $T\in \cL(H(\D))$ is \textit{not} injective. Moreover, \eqref{eq.Dim2} becomes
	\[
	{\rm dim}\, [T,X]=2{\rm dim}\, X
	\]
	and so $X\subsetneqq [T,X]$. As shown in the following result, for this setting the space $X$ is \textit{not dense} in $[T,X]$, the optimal domain $[T,X]$ is a complete semi-normed space and  $\|\cdot\|_{[T,X]}$ is \textit{not} a norm.
\end{remark}

We now consider the case when $E:=H(\D)$.

\begin{prop}\label{P.2} Suppose that $T\in \cL(H(\D))$ satisfies the following three conditions
	\begin{itemize}
		\item[\rm (a)] The operator  $T$ is not injective,
		\item[\rm (b)] $\Ker (T)\subseteq T(H(\D))$, and
		\item[\rm (c)] $\Ker (T)$ is finite dimensional. 
	\end{itemize}
	Let $X:=\Ker (T)$ be equipped with any norm $\|\cdot\|_X$. Then $X$ is a Banach space of analytic functions on $\D$ satisfying $T(X)\subseteq X$. Moreover, its optimal domain $[T,X]=T^{-1}(X)$ satisfies
	\begin{itemize}
		\item[\rm (i)] $X\subsetneqq [T,X]$,
		\item[\rm (ii)] $[T,X]$ is a  semi-normed space but, $\|\cdot\|_{[T,X]}$ is not a norm,
		\item[\rm (iii)] $X$ is not dense in $[T,X]$ and
		\item[\rm (iv)] $[T,X]$ is a complete semi-normed space.
	\end{itemize}
\end{prop}

\begin{proof} (i) Let $n={\rm dim}\, X$, with $n\geq 1$, and  $\{g_1,g_2,\ldots, g_n\}$ be a basis for $X$, i.e., $X={\rm span }\{g_j:\ 1\leq j\leq n\}$. By condition (b) for $T$ there exist functions $\{f_1,f_2,\ldots,f_n\}\subseteq H(\D)$ satisfying  $Tf_j=g_j$ for $1\leq j\leq n$.
	
	Since $g_j\not=0$, we have $f_j\not\in X$  for $1\leq j\leq n$. Moreover, $\{f_j\}_{j=1}^n\cup\{g_j\}_{j=1}^n$ are linearly independent in $H(\D)$. Clearly, $T(X)\subseteq X$.
	
	The lcH-topology $\tau_c$ in $H(\D)$ when restricted to $X$ is equivalent to any norm topology on $X$ (as ${\rm dim}\, X<\infty$). Fix a norm $\|\cdot\|_X$ in $X$, say,
	\[
	\left\|\sum_{j=1}^n\alpha_j g_j\right\|_X:=\max_{1\leq j\leq n}|\alpha_j|.
	\]
	Then $X\subseteq H(\D)$ with a continuous inclusion (actually the inclusion is an isomorphism into) and so $X$ is a Banach space of analytic functions on $\D$.
	
	The claim is that $T^{-1}(X):=\{f\in H(\D):\ Tf\in X\}$ is the linear subspace of $H(\D)$ given by 
	\begin{equation}\label{0.7}
		T^{-1}(X)=\Ker(T)+{\rm span}\{f_j:\ 1\leq j\leq n\}.
	\end{equation}
	Clearly, $\Ker(T)\subseteq T^{-1}(X)$. Furthermore, for  $1\leq j\leq n$, we have $f_j\in  T^{-1}(X)$   because $Tf_j=g_j\in X$. So, also ${\rm span}\{f_j:\ 1\leq j\leq n\}\subseteq T^{-1}(X)$. Since $T^{-1}(X)$ is a linear space, it follows that $\Ker(T)+{\rm span}\{f_j:\ 1\leq j\leq n\}\subseteq T^{-1}(X)$. Conversely, let $f\in T^{-1}(X)$. Then $Tf\in X$ and so $Tf=\sum_{j=1}^n\alpha_jg_j$ for some $(\alpha_1,\alpha_2,\ldots,\alpha_n)\in\C^n$. Define $h:=\sum_{j=1}^n\alpha_jf_j\in {\rm span}\{f_j:\ 1\leq j\leq n\}$, in which case 
	\[
	T(f-h)=Tf-Th=\sum_{j=1}^n\alpha_jg_j-\sum_{j=1}^n\alpha_jTf_j=\sum_{j=1}^n\alpha_jg_j-\sum_{j=1}^n\alpha_jg_j=0.
	\]
	Hence, $(f-h)\in \Ker (T)$. This implies that $f=(f-h)+h\in \Ker(T)+{\rm span}\{f_j:\ 1\leq j\leq n\}$. Therefore, $T^{-1}(X)\subseteq \Ker(T)+{\rm span}\{f_j:\ 1\leq j\leq n\}$. The claim is thereby established.
	
	Recall that 
	$[T,X]=\{f\in H(\D):\ Tf\in X\}=T^{-1}(X)$ and so $X\subseteq [T,X]$. It is routine to check that
	${\rm span}\{f_j:\ 1\leq j\leq n\}\cap X=\emptyset$ because $Tf_j=g_j\in X$, for $1\leq j\leq n$. It follows from \eqref{0.7} that
	\[
	X\subsetneqq [T,X].
	\]
	
	(ii)	Clearly, 
	\[
	\|f\|_{[T,X]}:=\|Tf\|_X,\quad f\in [T,X]
	\]
	specifies a seminorm in $[T,X]$ as $T$ is linear and $\|\cdot\|_X$ is a norm. Since $\|g_j\|_{[T,X]}=\|Tg_j\|_X=0$, for $1\leq j\leq n$, we see that $\|\cdot \|_{[T,X]}$ is  not a norm.
	
	(iii)	For every $g\in X$  note that
	\[
	\|g-f_1\|_{[T,X]}=\|Tg-Tf_1\|_X=\|g_1\|_X=1
	\]
	and so $X$ is not dense in $[T,X]$.
	
	(iv)	That the semi-normed space $[T,X]$ is complete is a consequence of  the following general result.  \qed
\end{proof}

Note that assumption (b) in Proposition \ref{P.2} is satisfied whenever $T\in \cL(H(\D))$ is surjective.

\begin{lemma}Every finite dimensional semi-normed space is complete.
\end{lemma}

\begin{proof} Let $(X,\|\cdot\|_X)$ be a finite dimensional semi-normed space. In the event that $\|\cdot\|_X$ is a norm, the result is well known. Indeed, every finite $n$-dimensional normed space is isomorphic to $\C^n$ (in the complex case) equipped with sup-norm $\|\cdot\|_\infty$.
	
	Suppose then that   $\|\cdot\|_X$ is not a norm. Set $X_0:=\Ker (\|\cdot\|_X)$. Then $X_0$ is a linear subspace of $X$. Since $X$ is finite dimensional,  there exists a linear subspace $X_1$ of $X$ such that $X=X_0\oplus X_1$. Clearly,  $X_1$ is  finite dimensional.  Moreover,  the semi-norm $\|\cdot\|_X$ when restricted to $X_1$ is a norm. Indeed, if $\|x\|_X=0$ for some $x\in X_1$, then $x\in X_0$ and so $x\in X_0\cap X_1=\{0\}$. Hence, $x=0$. Denote by  $\|\cdot\|_{X_1}$ the restriction of  $\|\cdot\|_{X}$ to $X_1$.  Therefore, $(X_1, \|\cdot\|_{X_1})$ is a finite dimensional normed space and so it is complete.
	Moreover, for every $x_0+x_1\in X_0\oplus X_1$, we have that $\|x_0+x_1\|_X\leq \|x_0\|_X +\|x_1\|_X=\|x_1\|_X$ and also that $\|x_1\|_X=\|x_0+x_1-x_0\|_X\leq \|x_0+x_1\|_X +\|x_0\|_X=\|x_0+x_1\|_X$. Consequently, for every $x_0+x_1\in X_0\oplus X_1$, we have  
	\begin{equation}\label{eq.UN}
		\|x_0+x_1\|_X=\|x_1\|_X .
	\end{equation}
	
	Let $(x^n_0+x^n_1)_{n\in\N}\subseteq  X_0\oplus X_1$  be a Cauchy sequence in $(X,\|\cdot\|_X)$. In view of \eqref{eq.UN} we have that $(x^n_1)_{n\in\N}\subseteq X_1$ is a Cauchy sequence in $(X_1,\|\cdot\|_{X_1})$. Since  $(X_1, \|\cdot\|_{X_1})$ is a complete normed space, there exists $x_1\in X_1$ such that $\|x_1^n-x_1\|_{X}=\|x_1^n-x_1\|_{X_1}\to 0$ as $n\to\ \infty$. In view of  \eqref{eq.UN}  it follows that $\|(x_0^n+x_1^n)-x_1\|_X=\|x_0^n+(x_1^n-x_1)\|_X=\|x_1^n-x_1\|_{X}\to 0$ as $n\to\infty$. This means that $(x_0^n+x_1^n)_{n\in\N}$ converges to $x_1$ in $(X\|\cdot\|_X)$. Therefore, we can conclude that $(X\|\cdot\|_X)$ is a complete semi-normed space.\qed
\end{proof}

The following two examples illustrate the applicability of  Proposition \ref{P.2}.

\begin{example}\label{Ex1}\rm  Let $P(z)=\sum_{j=0}^na_jz^j$  with $\{a_j\}_{j=0}^n\subseteq \C$ and $a_n\not=0$ be any complex polynomial of degree $n\geq 1$. Define the ordinary differential operator 
	\[
	T:=\sum_{j=0}^na_jD^j\quad \left({\rm with }\ D:=\frac{d}{dz}\right).
	\]
	The theory of ODE's implies that $T\in \cL(H(\D))$ is surjective, \cite[Ch. 2, Theorem 20]{CL}, and that 
	\[
	X:=\Ker(T)=\left\{f\in H(\D):\ \sum_{j=0}^na_jD^jf=0\right\}
	\]
	is an $n$-dimensional subspace of $H(\D)$, \cite[Ch. 2, Theorems 11 and 12]{CL}. So, 
	all the conditions (a)-(c) of the Proposition \ref{P.2} are satisfied.
\end{example}

\begin{example}\label{Ex2}\rm Define $V\in \cL(H(\D))$ by 
	\[
	(Vf)(z):=\int_0^zf(\xi)\,d\xi,\quad z\in \D,\ f\in H(\D),
	\]
	and set $T:=D+V$, where $D$ is the differentiation operator. Clearly, $T\in \cL(H(\D))$. The claim is that 
	\begin{equation}\label{0.9}
		\Ker(T)={\rm span}\{\cos (z)\}.
	\end{equation}
	Indeed,  $f\in \Ker(T)$ if and only if $f'(z)+\int_0^zf(\xi)\,d\xi=0$, for all $z\in \D$. It follows that  $f'(0)=0$ and, by differentiation, that $f$ satisfies $f''(z)+f(z)=0$, for all $z\in\D$. Two solutions of $f''(z)+f(z)=0$, for  $z\in\D$, which are linearly independent, are  $y_1(z):=\sin (z)$ and $y_2(z)=\cos (z)$, for $z\in\D$. Accordingly, $f(z):=\alpha \sin (z)+\beta\cos (z)$, for $z\in\D$, with $\alpha,\beta\in\C$. Then $f'(z)=\alpha \cos (z)-\beta \sin (z)$, for $z\in\D$, and the condition $f'(0)=0$  implies that $\alpha=0$. This establishes \eqref{0.9}.

	On the other hand, the optimal domain space is
	\begin{equation}\label{eq.NN}
		[T,X]={\rm span}\{\cos(z), \sin(z)+z\cos(z)\}.
	\end{equation}
	To establish  the validity of \eqref{eq.NN}, fix $f\in [T,X]$. Then $f\in H(\D)$ and $Tf\in \Ker (T)$. So, $(Tf)(z)=\gamma \cos (z)$ for some $\gamma\in \C$ and for $z\in\D$, i.e., $f'(z)+\int_0^zf(\xi)\,d\xi=\gamma \cos (z)$, for $z\in\D$. It follows  that $f'(0)=\gamma$ and, again by differentiation,  that
	\begin{equation}\label{eq.OD}
		f''(z)+f(z)=-\gamma \sin (z),\quad z\in\D.
	\end{equation}
	The theory of ODE's implies that the general solution of the non-homogeneous equation \eqref{eq.OD} is given by $f(z):=z(a\cos(z)+b\sin(z))+c_1\cos (z)+c_2\sin (z)$, for $z\in\D$, with $a, b, c_1, c_2\in\C$, where $g(z):=z(a\cos(z)+b\sin(z))$ is any solution of the equation $g''(z)+g(z)=-\gamma \sin (z)$, for $z\in\D$ (see \cite[Ch. 2, \S 10]{CL}). The requirement that $g$ satisfies $g''(z)+g(z)=-\gamma \sin (z)$, for $z\in\D$, implies  that $b=0$ and $a=\frac{\gamma}{2}$. Therefore, $g(z)=\frac{\gamma}{2}z\cos(z)$, for $z\in\D$, and hence, the general solution of \eqref{eq.OD} is given by $f(z)=\frac{\gamma}{2}z\cos(z)+c_1\cos (z)+c_2\sin (z)$, for $z\in\D$. Accordingly, $f'(z)=\frac{\gamma}{2}\cos (z)-\frac{\gamma}{2}z\sin(z)-c_1\sin (z)+c_2\cos(z)$, for $z\in\D$.   Then  the condition $f'(0)=\gamma$ yields  that $\gamma+c_2=\frac{\gamma}{2}$ from which it follows that $c_2=\frac{\gamma}{2}$.  Accordingly, $f(z)=\frac{\gamma}{2}(\sin(z)+z\cos(z))+c_1\cos (z)$, for $z\in\D$. This means that $f\in {\rm span}\{\cos(z), \sin(z)+z\cos(z)\}$, i.e, \eqref{eq.NN} is verified.
	
	Define $h(z):=\sin(z)+z\cos(z)$, for $z\in\D$. Then $h\in [T,X]\subseteq H(\D)$ in view of \eqref{eq.NN}. On the other hand, an easy computation shows that $Th=2\cos (z)\in X=\Ker(T)$. Since $X=\Ker(T)={\rm span}\{\cos (z)\}={\rm span}\{\frac{1}{2}Th\}$, it is clear that $\Ker(T)\subseteq T(H(\D))$. So, $T\in \cL(H(\D))$ satisfies condition (b) of the Proposition \ref{P.2}. Actually, the operator $T\in \cL(H(\D))$ is \textit{surjective} because the equation $Tf=g$, with $g\in H(\D)$ fixed, reduces to $f''(z)+f(z)=g'(z)$ and $f'(0)=g(0)$. Hence, the solution is $f(z)=k(z)+c_1\cos (z)+c_2\sin (z)$, for $z\in\D$, where $k\in H(\D)$ is given by
	\[
	k(z)=\cos(z)\int_0^z (-\sin\xi)g'(\xi)\,d\xi+\sin (z)\int_0^z\cos(\xi)g'(\xi)\,d\xi, \quad z\in\D,
	\]
	and $c_1,c_2\in\C$ are  determined by the condition $f'(0)=g(0)$ (\cite[Ch. 2, Theorem 20]{CL}).
\end{example}

\begin{remark}\rm  Consider $T:=D$; see Example \ref{Ex1} with $P(z)=z$, in which case $\Ker(T)={\rm span}\{\mathbbm{1}\}$. Note that 
	\[
	M(\Ker(T))={\rm span}\{z\}\not\subseteq \Ker(T)={\rm span}\{\mathbbm{1}\},
	\]
	where $M\in\cL(H(\D))$ is the multiplication operator given by 
	\[
	(Mf)(z):=zf(z),\quad z\in\D,\ f\in H(\D).
	\] 
	The fact that $M(\Ker(T))\not\subseteq \Ker(T)$ is a particular case of the following more general fact.
	
	Let $X\not=\{0\}$ be any Banach space of analytic functions on $\D$ such that $D(X)\subseteq X$. The closed graph theorem yields that $D\in \cL(X)$. For such an $X$ it cannot be the case that also $M(X)\subseteq X$, because then $M\in \cL(X)$ and, via direct calculation, $D$ and $M$ satisfy the identity 
	\[
	DM-MD=I 
	\]
	in $\cL(X)$. This is impossible. Indeed, applying \cite[(1)]{K} for the Banach algebra $\cL(X)$ with $A:=D$, $T:=M$ and $Q:=I$, implies that the operator $I$ is quasi-nilpotent, i.e., $1=(\|I^n\|_{X\to X})^{1/n}\to \infty$ as $n\to\infty$, which  is a contradiction; see also \cite{S}. Similarly, if $Y\not=\{0\}$ is a  Banach space of analytic functions on $\D$ such that $M(Y)\subseteq Y$, then necessarily $D(Y)\not\subseteq Y$. 
	
	For the particular space $X=\Ker (D)=\Ker (T)$ mentioned above, which does  satisfy $D(X)\subseteq X$, it  \textit{must} be the case that $M(X)\not\subseteq X$. The conclusion is:
	
\textit{	Let $X\not=\{0\}$ be any Banach space of analytic functions on $\D$ such that $D(X)\subseteq X$. Then necessarily $M(X)\not\subseteq X$.}
\end{remark}

\vspace{.1cm}

\textbf{Acknowledgements.} The research of A.A. Albanese was partially supported by GNAMPA of  INDAM.
The research of J. Bonet was partially supported by project PID2024-162128NB-I00 funded by MICIU/AEI /10.13039/501100011033 / FEDER, UE and project CIAICO/2023/242 funded by Conselleria de Innovación, Universidades, Ciencia y Sociedad Digital de la Generalitat Valenciana.

\medskip
\noindent
\textbf{Data Availability Statement} Not applicable
\medskip

\noindent
{\textbf{Declarations }}

\medskip
\noindent
\textbf{Conflict of interest}
The authors have no Conflict of interest to declare that	are relevant to the content of this article.


\bibliographystyle{plain}

\end{document}